\documentclass{amsart}
\usepackage[T1]{fontenc}

\usepackage{amsmath}
\usepackage{amssymb}
\usepackage{mathrsfs}

\usepackage{float}
\usepackage[dvips]{graphicx,color}
 
\usepackage{ascmac}

\usepackage{tabularx}

\usepackage{tikz}
\usetikzlibrary{calc,decorations.markings}

\usepackage{amsthm}
\theoremstyle{definition}
\newtheorem{THM}{Theorem}
\newtheorem{LEM}[THM]{Lemma}
\newtheorem{PROP}[THM]{Proposition}

\newtheorem{DEF}[THM]{Definition}

\newtheorem{QUE}[THM]{Question}
\newtheorem*{THM*}{Theorem}
\newtheorem*{LEM*}{Lemma}
\newtheorem*{PROP*}{Proposition}
\newtheorem*{COR*}{Corollary}
\newtheorem*{DEF*}{Definition}
\newtheorem*{RMK*}{Remark}
\newtheorem*{EX*}{Example}

\numberwithin{figure}{section}
\numberwithin{equation}{section}
\numberwithin{THM}{section}

\title{A $q$-series identity via the $\mathfrak{sl}_3$ colored Jones polynomials for the $(2,2m)$-torus link}
\author{Wataru Yuasa}
\address{Department of Mathematics\\
  Tokyo Institute of Technology\\
  2-12-1 Ookayama, Meguro-ku, Tokyo 152-8551, Japan}
\email[]{yuasa.w.aa@m.titech.ac.jp}

\begin{document}
\begin{abstract}
The colored Jones polynomial is a $q$-polynomial invariant of links colored by irreducible representations of a simple Lie algebra. 
A $q$-series called a tail is obtained as the limit of the $\mathfrak{sl}_2$ colored Jones polynomials $\{J_n(K;q)\}_n$ for some link $K$, for example, an alternating link. 
For the $\mathfrak{sl}_3$ colored Jones polynomials, 
the existence of a tail is unknown. 
We give two explicit formulas of the tail of the $\mathfrak{sl}_3$ colored Jones polynomials colored by $(n,0)$ for the $(2,2m)$-torus link. 
These two expressions of the tail provide an identity of $q$-series. 
This is a knot-theoretical generalization of the Andrews-Gordon identities for the Ramanujan false theta function.
\end{abstract}
\maketitle

\tikzset{->-/.style={decoration={
  markings,
  mark=at position #1 with {\arrow[black,thin]{>}}},postaction={decorate}}}
\tikzset{-<-/.style={decoration={
  markings,
  mark=at position #1 with {\arrow[black,thin]{<}}},postaction={decorate}}}

\tikzset{
    triple/.style args={[#1] in [#2] in [#3]}{
        #1,preaction={preaction={draw,#3},draw,#2}
    }
}

\section{Introduction}
The Rogers-Ramanujan identities were first discovered and proved by Rogers~\cite{Rogers1894}. 
After that, these identities appeared in Ramanujan's first letter to Hardy without proof. 
To this day, 
the Rogers-Ramanujan identities have been generalized by many people and have been proved in many ways.
One of these generalizations is the Andrews-Gordon identities.
\begin{THM}[The Andrews-Gordon identities for the Ramanujan theta function~\cite{Andrews74}]\label{theta}
\[ f(-q^{2m},-q)=(q;q)_\infty\sum_{k_m\leq\dots\leq k_2\leq k_1}\frac{q^{\sum_{j=1}^{m-1}k_j(k_j+1)}}{\prod_{j=1}^{m-1}(q;q)_{k_j-k_{j+1}}},
\]
where $m>0$ and $f(a,b)=\sum_{i=0}^\infty a^{\frac{i(i+1)}{2}}b^{\frac{i(i-1)}{2}}+\sum_{i=1}^\infty a^{\frac{i(i-1)}{2}}b^{\frac{i(i+1)}{2}}$ is Ramanujan's general theta function (see, for example, \cite{AndrewsBerndt05}).
\end{THM}
In knot theory, the Andrews-Gordon identities appear through the ($\mathfrak{sl}_2$) colored Jones polynomial. 
The colored Jones polynomial is a $q$-polynomial invariant of knots and links.
Dasbach and Lin showed the stability of some coefficients of the colored Jones polynomials for alternating knots in \cite{DasbachLin06,DasbachLin07}. 
They conjectured that the $k$-th coefficient of the $n+1$ dimensional colored Jones polynomial $J_n(K;q)$ for an alternating knot $K$ coincides up to sign for $n\geq k$. 
This suggests the existence of $q$-series $\sum_{k=0}^\infty a_kq^k$ for an alternating link $K$ called the {\em tail}. The coefficient $a_k$ is given by the $k$-th coefficient of $J_n(K;q)$ with $n\geq k$.
Armond~\cite{Armond13} proved the existence of tails for adequate links,
in particular for alternating knots.
Independently, for alternating knots, Garoufalidis and L{\^e}~\cite{GaroufalidisLe15} showed a more general stability of coefficients of $J_n(K;q)$. 
In~\cite{ArmondDasbach11}, 
Armond and Dasbach showed two explicit formulas for the tails of the colored Jones polynomial for the $(2,2m+1)$-torus knot. 
These two formulas conclude the left- and right-hand side of the Andrews-Gordon identities for the Ramanujan theta function. 
They used the formula for the colored Jones polynomial of the $(2,2m+1)$-torus knot in~\cite{Morton95} and obtained by the method of Armond~\cite{Armond14}. 
Hajij~\cite{Hajij16} also showed the Andrews-Gordon identities for the Ramanujan theta function and the Ramanujan false theta function using two expressions of the tail of the $(2,2m+1)$-torus knot and the $(2,2m)$-torus link, respectively. 
\begin{THM}[The Andrews-Gordon identities for the Ramanujan false theta function~\cite{Hajij16}]\label{false}
\[
\Psi(q^{2m-1},q)=(q;q)_\infty\sum_{k_{m-1}\leq\dots\leq k_2\leq k_1}\frac{q^{\sum_{j=1}^{m-1}k_j(k_j+1)}}{(q;q)_{k_{m-1}}^2\prod_{j=1}^{m-2}(q;q)_{k_j-k_{j+1}}},
\]
where $m>1$ and $\Psi(a,b)=\sum_{i=0}^\infty a^{\frac{i(i+1)}{2}}b^{\frac{i(i-1)}{2}}-\sum_{i=0}^\infty a^{\frac{i(i-1)}{2}}b^{\frac{i(i+1)}{2}}$ is Ramanujan's general false theta function (see, for example, \cite{McLaughlinSillsZimmer09}).
\end{THM}
He used formulas for the colored Jones polynomial obtained by using a colored trivalent graph representation of the Kauffman bracket skein elements (see, for example, \cite{KauffmanLins94}) and the Kauffman bracket bubble skein expansion formula in~\cite{Hajij14A}.

The main subject of this paper is the tail of the colored $\mathfrak{sl}_3$ Jones polynomials for the $(2,2m)$-torus link. 
The existence of tails of the $\mathfrak{sl}_3$ colored Jones polynomials for any class of knots and links is unknown. 
Moreover, 
there is no example of an explicit formula for a tail of the $\mathfrak{sl}_3$ colored Jones polynomial. 
We will give two explicit formulas of ``the $\mathfrak{sl}_3$ tail'' for the $(2,2m)$-torus link. 
As the result, 
the following $q$-series identity holds. 
\setcounter{section}{4}
\setcounter{THM}{2}
\begin{THM}
\footnotesize
\[
\sum_{i=0}^\infty q^{-2i}q^{m(i^2+2i)}\frac{(1-q^{i+1})^3(1+q^{i+1})}{1-q}=(q)_{\infty}\sum_{0\leq k_m\leq\dots\leq k_2\leq k_1}\frac{q^{-2k_m}q^{\sum_{j=1}^m(k_i^2+2k_i)}}{(q)_{k_m}^2(q)_{k_1-k_2}\dots(q)_{k_{m-1}-k_m}}.\]
\normalsize
\end{THM}
\setcounter{section}{1}
\setcounter{THM}{2}
The right-hand side of the above identity is obtained from the full twist formula in~\cite{Yuasa16}. 
The left-hand side is obtained by using of colored trivalent graph presentation of $A_2$ webs.

The paper is organized as follows. 
We first review definitions and formulas related to the $A_2$ web space in section~2. 
The $A_2$ web space is a generalization of the Kauffman bracket skein module. 
In section~3, we introduce a method to represent some types of $A_2$ webs using colored trivalent graphs. 
Furthermore, 
we give values of some $\theta$-graphs and quantum $6j$ symbols. 
As an application, 
we explicitly give the $\mathfrak{sl}_3$ colored Jones polynomial for a $2$-bridge link.
In section~4, 
we derive two $q$-series, that is ``the $\mathfrak{sl}_3$ tails'', by using explicit formulas for the $\mathfrak{sl}_3$ colored Jones polynomial for the $(2,2m)$-torus link in section~3 and in~\cite{Yuasa16}.

\section{The $A_2$ web space and some formulas}
The skein theory has been developed with a quantum representation of Lie algebra $\mathfrak{g}$. 
If $\mathfrak{g}=\mathfrak{sl}_2$, the skein theory consists of the Kauffman bracket skein module, which is called the Temperley-Lieb algebra and the $A_1$ web space, and the Kauffman bracket. 
Kuperberg constructed the skein theory for Lie algebras of rank $2$ in~\cite{Kuperberg94, Kuperberg96}. 
In this section, 
we review definitions of the $A_2$ web space, the $A_2$ bracket and the $A_2$ clasp.

Let $\varepsilon=(\varepsilon_1,\varepsilon_2,\dots,\varepsilon_m)$ be an $m$-tuple of signs ${+}$ or ${-}$. 
Let $D_\varepsilon$ denote the unit disk with signed marked points $\{\exp(2\pi\sqrt{-1}/m)^{j-1}\mid j=1,2,\dots,m \}$ on its boundary. 
The sign of $\exp(2\pi\sqrt{-1}/m)^{j-1}$ is given by $\varepsilon_j$ for $j=1,2,\dots,m$.
A {\em bipartite uni-trivalent graph} $G$ is a directed graph such that each vertex is either trivalent or univalent and the vertices are divided into the sinks and the sources. 
A sink (resp. source) is a vertex such that all edges adjoining to the vertex point into (resp. away from) it.
A {\em bipartite trivalent graph} $G$ in $D_{\varepsilon}$ is an embedding of a uni-trivalent graph into $D_\varepsilon$ such that any vertex $v$ has the following neighborhoods:
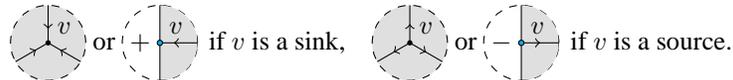
\[
\tikz[baseline=-.6ex]{
\draw [thin, dashed, fill=lightgray!50] (0,0) circle [radius=.5];
\draw[-<-=.5] (0:0) -- (90:.5);
\draw[-<-=.5] (0:0) -- (210:.5);
\draw[-<-=.5] (0:0) -- (-30:.5);
\node (v) at (0,0) [above right]{$v$};
\fill (0,0) circle [radius=1pt];
}
\text{\ or\ }
\tikz[baseline=-.6ex]{
\draw [thin, dashed] (0,0) circle [radius=.5];
\clip (0,0) circle [radius=.5];
\draw [thin, fill=lightgray!50] (0,-.5) rectangle (.5,.5);
\draw[-<-=.5] (0,0)--(.5,0);
\node (p) at (0,0) [left]{${+}$};
\node at (0,0) [above right]{$v$};
\draw [fill=cyan] (0,0) circle [radius=1pt];
}\text{\, if $v$ is a sink,}
\quad\tikz[baseline=-.6ex]{
\draw [thin, dashed, fill=lightgray!50] (0,0) circle [radius=.5];
\draw[->-=.5] (0:0) -- (90:.5);
\draw[->-=.5] (0:0) -- (210:.5);
\draw[->-=.5] (0:0) -- (-30:.5);
\node (v) at (0,0) [above right]{$v$};
\fill (0,0) circle [radius=1pt];
}
\text{\ or\ }
\tikz[baseline=-.6ex]{
\draw [thin, dashed] (0,0) circle [radius=.5];
\clip (0,0) circle [radius=.5];
\draw [thin, fill=lightgray!50] (0,-.5) rectangle (.5,.5);
\draw[->-=.5] (0,0)--(.5,0);
\node (p) at (0,0) [left]{${-}$};
\node at (0,0) [above right]{$v$};
\draw [fill=cyan] (0,0) circle [radius=1pt];
}\text{\, if $v$ is a source.}
\]
An {\em $A_2$ basis web} is the boundary-fixing isotopy class of a bipartite trivalent graph $G$ in $D_{\varepsilon}$, 
where any internal face of $D_{\varepsilon}\setminus G$ has at least six sides. 
Let us denote $B_\varepsilon$ the set of $A_2$ basis webs in $D_{\varepsilon}$.
For example,
$B_{(+,-,+,-,+,-)}$ has the following $A_2$ basis webs:
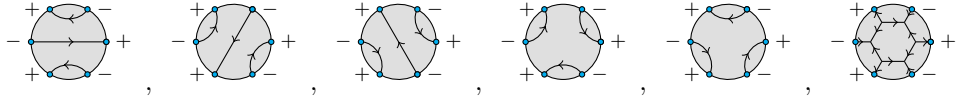
\[
\,\begin{tikzpicture}
\draw [thin, fill=lightgray!50] (0,0) circle [radius=.5];
\draw[-<-=.5] (0:.5) -- (180:.5);
\draw[->-=.5] (60:.5) to[out=-120, in=-60] (120:.5);
\draw[-<-=.5] (240:.5) to[out=60, in=120] (300:.5);
\foreach \i in {0,1,...,6} \draw[fill=cyan] ($(0,0) !1! \i*60:(.5,0)$) circle [radius=1pt];
\node at (0:.5) [right]{$\scriptstyle{+}$};
\node at (60:.5) [right]{$\scriptstyle{-}$};
\node at (120:.5) [left]{$\scriptstyle{+}$};
\node at (180:.5) [left]{$\scriptstyle{-}$};
\node at (240:.5) [left]{$\scriptstyle{+}$};
\node at (300:.5) [right]{$\scriptstyle{-}$};
\end{tikzpicture}\,,
\,\begin{tikzpicture}
\begin{scope}[rotate=60]
\draw [thin, fill=lightgray!50] (0,0) circle [radius=.5];
\draw[->-=.5] (0:.5) -- (180:.5);
\draw[-<-=.5] (60:.5) to[out=-120, in=-60] (120:.5);
\draw[->-=.5] (240:.5) to[out=60, in=120] (300:.5);
\foreach \i in {0,1,...,6} \draw[fill=cyan] ($(0,0) !1! \i*60:(.5,0)$) circle [radius=1pt];
\end{scope}
\node at (0:.5) [right]{$\scriptstyle{+}$};
\node at (60:.5) [right]{$\scriptstyle{-}$};
\node at (120:.5) [left]{$\scriptstyle{+}$};
\node at (180:.5) [left]{$\scriptstyle{-}$};
\node at (240:.5) [left]{$\scriptstyle{+}$};
\node at (300:.5) [right]{$\scriptstyle{-}$};
\end{tikzpicture}\,,
\,\begin{tikzpicture}
\begin{scope}[rotate=-60]
\draw [thin, fill=lightgray!50] (0,0) circle [radius=.5];
\draw[->-=.5] (0:.5) -- (180:.5);
\draw[-<-=.5] (60:.5) to[out=-120, in=-60] (120:.5);
\draw[->-=.5] (240:.5) to[out=60, in=120] (300:.5);
\foreach \i in {0,1,...,6} \draw[fill=cyan] ($(0,0) !1! \i*60:(.5,0)$) circle [radius=1pt];
\end{scope}
\node at (0:.5) [right]{$\scriptstyle{+}$};
\node at (60:.5) [right]{$\scriptstyle{-}$};
\node at (120:.5) [left]{$\scriptstyle{+}$};
\node at (180:.5) [left]{$\scriptstyle{-}$};
\node at (240:.5) [left]{$\scriptstyle{+}$};
\node at (300:.5) [right]{$\scriptstyle{-}$};
\end{tikzpicture}\,,
\,\begin{tikzpicture}
\draw [thin, fill=lightgray!50] (0,0) circle [radius=.5];
\draw[-<-=.5] (0:.5) to[out=180, in=-120] (60:.5);
\draw[rotate=120, -<-=.5] (0:.5) to[out=180, in=-120] (60:.5);
\draw[rotate=240, -<-=.5] (0:.5) to[out=180, in=-120] (60:.5);
\foreach \i in {0,1,...,6} \draw[fill=cyan] ($(0,0) !1! \i*60:(.5,0)$) circle [radius=1pt];
\node at (0:.5) [right]{$\scriptstyle{+}$};
\node at (60:.5) [right]{$\scriptstyle{-}$};
\node at (120:.5) [left]{$\scriptstyle{+}$};
\node at (180:.5) [left]{$\scriptstyle{-}$};
\node at (240:.5) [left]{$\scriptstyle{+}$};
\node at (300:.5) [right]{$\scriptstyle{-}$};
\end{tikzpicture}\,,
\,\begin{tikzpicture}
\begin{scope}[rotate=60]
\draw [thin, fill=lightgray!50] (0,0) circle [radius=.5];
\draw[->-=.5] (0:.5) to[out=180, in=-120] (60:.5);
\draw[rotate=120, ->-=.5] (0:.5) to[out=180, in=-120] (60:.5);
\draw[rotate=240, ->-=.5] (0:.5) to[out=180, in=-120] (60:.5);
\foreach \i in {0,1,...,6} \draw[fill=cyan] ($(0,0) !1! \i*60:(.5,0)$) circle [radius=1pt];
\end{scope}
\node at (0:.5) [right]{$\scriptstyle{+}$};
\node at (60:.5) [right]{$\scriptstyle{-}$};
\node at (120:.5) [left]{$\scriptstyle{+}$};
\node at (180:.5) [left]{$\scriptstyle{-}$};
\node at (240:.5) [left]{$\scriptstyle{+}$};
\node at (300:.5) [right]{$\scriptstyle{-}$};
\end{tikzpicture}\,,
\,\begin{tikzpicture}
\draw [thin, fill=lightgray!50] (0,0) circle [radius=.5];
\draw[-<-=.5] (0:.5) -- (0:.3);
\draw[rotate=60, ->-=.5] (0:.5) -- (0:.3);
\draw[rotate=120, -<-=.5] (0:.5) -- (0:.3);
\draw[rotate=180, ->-=.5] (0:.5) -- (0:.3);
\draw[rotate=240, -<-=.5] (0:.5) -- (0:.3);
\draw[rotate=300, ->-=.5] (0:.5) -- (0:.3);
\draw[->-=.5] (0:.3) -- (60:.3);
\draw[rotate=60, -<-=.5] (0:.3) -- (60:.3);
\draw[rotate=120, ->-=.5] (0:.3) -- (60:.3);
\draw[rotate=180, -<-=.5] (0:.3) -- (60:.3);
\draw[rotate=240, ->-=.5] (0:.3) -- (60:.3);
\draw[rotate=300, -<-=.5] (0:.3) -- (60:.3);
\foreach \i in {0,1,...,6} \draw[fill=cyan] ($(0,0) !1! \i*60:(.5,0)$) circle [radius=1pt];
\node at (0:.5) [right]{$\scriptstyle{+}$};
\node at (60:.5) [right]{$\scriptstyle{-}$};
\node at (120:.5) [left]{$\scriptstyle{+}$};
\node at (180:.5) [left]{$\scriptstyle{-}$};
\node at (240:.5) [left]{$\scriptstyle{+}$};
\node at (300:.5) [right]{$\scriptstyle{-}$};
\end{tikzpicture}\,.
\]
The {\em $A_2$ web space $W_\varepsilon$} is the $\mathbb{Q}(q^{\frac{1}{6}})$-vector space spanned by $B_\varepsilon$. 
A {\em tangled trivalent graph diagram} in $D_\varepsilon$ is an immersed bipartite uni-trivalent graph in $D_\varepsilon$ whose intersection points are only transverse double points of edges with crossing data 
\,\tikz[baseline=-.6ex, scale=.8]{
\draw [thin, dashed, fill=lightgray!50] (0,0) circle [radius=.5];
\draw[->-=.8] (-45:.5) -- (135:.5);
\draw[->-=.8, lightgray!50, double=black, double distance=0.4pt, ultra thick] (-135:.5) -- (45:.5);
}\, or 
\,\tikz[baseline=-.6ex, scale=.8]{
\draw [thin, dashed, fill=lightgray!50] (0,0) circle [radius=.5];
\draw[->-=.8] (-135:.5) -- (45:.5);
\draw[->-=.8, lightgray!50, double=black, double distance=0.4pt, ultra thick] (-45:.5) -- (135:.5);
}\, .
Tangled trivalent graph diagrams $G$ and $G'$ are regularly isotopic if $G$ is obtained from $G'$ by a finite sequence of boundary-fixing isotopies and Reidemeister moves, see Figure~\ref{Reidemeister}, with some direction of edges.
\begin{figure}
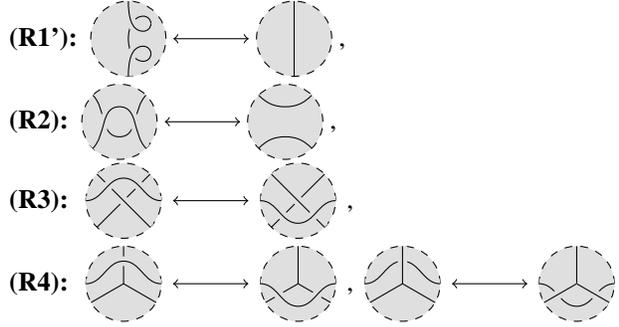

\begin{description}
\item[(R1')]
\tikz[baseline=-.6ex]{
\draw [thin, dashed, fill=lightgray!50] (0,0) circle [radius=.5];
\draw (.3,-.2)
to[out=south, in=east] (.2,-.3)
to[out=west, in=south] (.0,.0)
to[out=north, in=west] (.2,.3)
to[out=east, in=north] (.3,.2);
\draw[lightgray!50, double=black, double distance=0.4pt, ultra thick] (0,-.5) 
to[out=north, in=west] (.2,-.1)
to[out=east, in=north] (.3,-.2);
\draw[lightgray!50, double=black, double distance=0.4pt, ultra thick] (.3,.2)
to[out=south, in=east] (.2,.1)
to[out=west, in=south] (0,.5);
}
\tikz[baseline=-.6ex]{
\draw [<->, xshift=1.5cm] (1,0)--(2,0);
}
\tikz[baseline=-.6ex]{
\draw[xshift=3cm, thin, dashed, fill=lightgray!50] (0,0) circle [radius=.5];
\draw[xshift=3cm] (90:.5) to (-90:.5);
}\ ,
\item[(R2)]
\tikz[baseline=-.6ex]{
\draw [thin, dashed, fill=lightgray!50] (0,0) circle [radius=.5];
\draw (135:.5) to [out=south east, in=west](0,-.2) to [out=east, in=south west](45:.5);
\draw [lightgray!50, double=black, double distance=0.4pt, ultra thick](-135:.5) to [out=north east, in=left](0,.2) to [out=right, in=north west] (-45:.5);
}
\tikz[baseline=-.6ex]{
\draw [<->, xshift=1.5cm] (1,0)--(2,0);
}
\tikz[baseline=-.6ex]{
\draw[xshift=3cm, thin, dashed, fill=lightgray!50] (0,0) circle [radius=.5];
\draw[xshift=3cm] (135:.5) to [out=south east, in=west](0,.2) to [out=east, in=south west](45:.5);
\draw[xshift=3cm] (-135:.5) to [out=north east, in=left](0,-.2) to [out=right, in=north west] (-45:.5);
}\ ,
\item[(R3)]
\tikz[baseline=-.6ex]{
\draw [thin, dashed, fill=lightgray!50] (0,0) circle [radius=.5];
\draw (-135:.5) -- (45:.5);
\draw [lightgray!50, double=black, double distance=0.4pt, ultra thick] (135:.5) -- (-45:.5);
\draw[lightgray!50, double=black, double distance=0.4pt, ultra thick](180:.5) to [out=right, in=left](0,.3) to [out=right, in=left] (-0:.5);
}
\tikz[baseline=-.6ex]{
\draw[<->, xshift=1.5cm] (1,0)--(2,0);
}
\tikz[baseline=-.6ex]{
\draw [xshift=3cm, thin, dashed, fill=lightgray!50] (0,0) circle [radius=.5];
\draw [xshift=3cm] (-135:.5) -- (45:.5);
\draw [xshift=3cm, lightgray!50, double=black, double distance=0.4pt, ultra thick] (135:.5) -- (-45:.5);
\draw[xshift=3cm, lightgray!50, double=black, double distance=0.4pt, ultra thick](180:.5) to [out=right, in=left](0,-.3) to [out=right, in=left] (-0:.5);
}\ ,
\item[(R4)]
\tikz[baseline=-.6ex]{
\draw [thin, dashed, fill=lightgray!50] (0,0) circle [radius=.5];
\draw (0:0) -- (90:.5); 
\draw (0:0) -- (210:.5); 
\draw (0:0) -- (-30:.5);
\draw[lightgray!50, double=black, double distance=0.4pt, ultra thick](180:.5) to [out=right, in=left](0,.3) to [out=right, in=left] (-0:.5);
}
\tikz[baseline=-.6ex]{
\draw[<->, xshift=1.5cm] (1,0)--(2,0);
}
\tikz[baseline=-.6ex]{
\draw [thin, dashed, fill=lightgray!50] (0,0) circle [radius=.5];
\draw (0:0) -- (90:.5); 
\draw (0:0) -- (210:.5); 
\draw (0:0) -- (-30:.5);
\draw[lightgray!50, double=black, double distance=0.4pt, ultra thick](180:.5) to [out=right, in=left](0,-.3) to [out=right, in=left] (-0:.5);
}\ , 
\tikz[baseline=-.6ex]{
\draw [thin, dashed, fill=lightgray!50] (0,0) circle [radius=.5];
\draw[lightgray!50, double=black, double distance=0.4pt, ultra thick](180:.5) to [out=right, in=left](0,.3) to [out=right, in=left] (-0:.5);
\draw[lightgray!50, double=black, double distance=0.4pt, ultra thick] (0:0) -- (90:.5); 
\draw (0:0) -- (210:.5); 
\draw (0:0) -- (-30:.5);
}
\tikz[baseline=-.6ex]{
\draw[<->, xshift=1.5cm] (1,0)--(2,0);
}
\tikz[baseline=-.6ex]{
\draw [thin, dashed, fill=lightgray!50] (0,0) circle [radius=.5];
\draw[lightgray!50, double=black, double distance=0.4pt, ultra thick](180:.5) to [out=right, in=left](0,-.3) to [out=right, in=left] (-0:.5);
\draw (0:0) -- (90:.5); 
\draw[lightgray!50, double=black, double distance=0.4pt, ultra thick] (0:0) -- (210:.5); 
\draw[lightgray!50, double=black, double distance=0.4pt, ultra thick] (0:0) -- (-30:.5);
}.
\end{description}
\caption{The Reidemeister moves for tangled trivalent graph diagrams}
\label{Reidemeister}
\end{figure}

{\em Tangled trivalent graphs} in $D_\varepsilon$ are regular isotopy classes of tangled trivalent graph diagrams in $D_\varepsilon$. 
We denote $T_\varepsilon$ the set of tangled trivalent graphs in $D_\varepsilon$.

\begin{DEF}[The $A_2$ bracket~\cite{Kuperberg96}]
We define a $\mathbb{Q}(q^{\frac{1}{6}})$-linear map $\langle\,\cdot\,\rangle_3\colon\mathbb{Q}(q^{\frac{1}{6}})T_\varepsilon\to W_\varepsilon$ by the following.
\begin{itemize}
\item 
$\Big\langle\,\tikz[baseline=-.6ex, scale=0.8]{
\draw [thin, dashed, fill=lightgray!50] (0,0) circle [radius=.5];
\draw[->-=.8] (-45:.5) -- (135:.5);
\draw[->-=.8, lightgray!50, double=black, double distance=0.4pt, ultra thick] (-135:.5) -- (45:.5);
}\,\Big\rangle_{\! 3}
=
q^{\frac{1}{3}}\Big\langle\,\tikz[baseline=-.6ex, scale=0.8]{
\draw[thin, dashed, fill=lightgray!50] (0,0) circle [radius=.5];
\draw[->-=.5] (-45:.5) to [out=north west, in=south](.2,0) to [out=north, in=south west](45:.5);
\draw[->-=.5] (-135:.5) to [out=north east, in=south](-.2,0) to [out=north, in=south east] (135:.5);
}\,\Big\rangle_{\! 3}
-
q^{-\frac{1}{6}}\Big\langle\,\tikz[baseline=-.6ex, scale=0.8]{
\draw[thin, dashed, fill=lightgray!50] (0,0) circle [radius=.5];
\draw[->-=.5] (-45:.5) -- (0,-.2);
\draw[->-=.5] (-135:.5) -- (0,-.2);
\draw[-<-=.5] (0,-.2) -- (0,.2);
\draw[-<-=.5] (45:.5) -- (0,.2);
\draw[-<-=.5] (135:.5) -- (0,.2);
}\,\Big\rangle_{\! 3}
$,
\item 
$\Big\langle\,\tikz[baseline=-.6ex, scale=0.8]{
\draw [thin, dashed, fill=lightgray!50] (0,0) circle [radius=.5];
\draw[->-=.8] (-135:.5) -- (45:.5);
\draw[->-=.8, lightgray!50, double=black, double distance=0.4pt, ultra thick] (-45:.5) -- (135:.5);
}\,\Big\rangle_{\! 3}
=
q^{-\frac{1}{3}}\Big\langle\,\tikz[baseline=-.6ex, scale=0.8]{
\draw[thin, dashed, fill=lightgray!50] (0,0) circle [radius=.5];
\draw[->-=.5] (-45:.5) to [out=north west, in=south](.2,0) to [out=north, in=south west](45:.5);
\draw[->-=.5] (-135:.5) to [out=north east, in=south](-.2,0) to [out=north, in=south east] (135:.5);
}\,\Big\rangle_{\! 3}
-
q^{\frac{1}{6}}\Big\langle\,\tikz[baseline=-.6ex, scale=0.8]{
\draw[thin, dashed, fill=lightgray!50] (0,0) circle [radius=.5];
\draw[->-=.5] (-45:.5) -- (0,-.2);
\draw[->-=.5] (-135:.5) -- (0,-.2);
\draw[-<-=.5] (0,-.2) -- (0,.2);
\draw[-<-=.5] (45:.5) -- (0,.2);
\draw[-<-=.5] (135:.5) -- (0,.2);
}\,\Big\rangle_{\! 3}
$,
\item 
$\Big\langle\,\tikz[baseline=-.6ex, rotate=90, scale=0.8]{
\draw[thin, dashed, fill=lightgray!50] (0,0) circle [radius=.5];
\draw[->-=.6] (-45:.5) -- (-45:.3);
\draw[->-=.6] (-135:.5) -- (-135:.3);
\draw[-<-=.6] (45:.5) -- (45:.3);
\draw[->-=.6] (135:.5) -- (135:.3);
\draw[->-=.5] (45:.3) -- (135:.3);
\draw[-<-=.5] (-45:.3) -- (-135:.3);
\draw[->-=.5] (45:.3) -- (-45:.3);
\draw[-<-=.5] (135:.3) -- (-135:.3);
}\,\Big\rangle_{\! 3}
=
\Big\langle\,\tikz[baseline=-.6ex, rotate=90, scale=0.8]{
\draw[thin, dashed, fill=lightgray!50] (0,0) circle [radius=.5];
\draw[->-=.5] (-45:.5) to [out=north west, in=south](.2,0) to [out=north, in=south west](45:.5);
\draw[-<-=.5] (-135:.5) to [out=north east, in=south](-.2,0) to [out=north, in=south east] (135:.5);
}\,\Big\rangle_{\! 3}
+
\Big\langle\,\tikz[rotate=90, baseline=-.6ex, rotate=90, scale=0.8]{
\draw[thin, dashed, fill=lightgray!50] (0,0) circle [radius=.5];
\draw[-<-=.5] (-45:.5) to [out=north west, in=south](.2,0) to [out=north, in=south west](45:.5);
\draw[->-=.5] (-135:.5) to [out=north east, in=south](-.2,0) to [out=north, in=south east] (135:.5);
}\,\Big\rangle_{\! 3}
$,
\item 
$\Big\langle\,\tikz[baseline=-.6ex, rotate=-90, scale=0.8]{
\draw[thin, dashed, fill=lightgray!50] (0,0) circle [radius=.5];
\draw[->-=.5] (0,-.5) -- (0,-.25);
\draw[->-=.5] (0,.25) -- (0,.5);
\draw[-<-=.5] (0,-.25) to [out=60, in=120, relative](0,.25);
\draw[-<-=.5] (0,-.25) to [out=-60, in=-120, relative](0,.25);
}\,\Big\rangle_{\! 3}
=
\left[2\right]\Big\langle\,\tikz[baseline=-.6ex, rotate=-90, scale=0.8]{
\draw[thin, dashed, fill=lightgray!50] (0,0) circle [radius=.5];
\draw[->-=.5] (0,-.5) -- (0,.5);
}\,\Big\rangle_{\! 3}
$,
\item 
$
\Big\langle G\sqcup
\,\tikz[baseline=-.6ex, scale=0.8]{
\draw[thin, dashed, fill=lightgray!50] (0,0) circle [radius=.5];
\draw[->-=.5] (0,0) circle [radius=.3];
}\,\Big\rangle_{\! 3}
=
\left[3\right] \langle G\rangle_{3},
$
\end{itemize}
where $\left[n\right]=\frac{q^{\frac{n}{2}}-q^{-\frac{n}{2}}}{q^{\frac{1}{2}}-q^{-\frac{1}{2}}}$ is a quantum integer.
\end{DEF}
We remark that this map is invariant under the Reidemeister moves for tangled trivalent graphs.

We next consider the $A_2$ web space $W_{n^{+}+n^{-}}=W_{({+},{+},\dots,{+},{-},{-},\dots,{-})}$ whose first $n$ marked points are decorated with ${+}$ and next $n$ marked points are decorated with ${-}$.
We define $A_2$ clasps 
$\tikz[baseline=-.6ex]{
\draw[->-=.8] (-.5,0) -- (.5,0);
\draw[fill=white] (-.1,-.3) rectangle (.1,.3);
\node at (.1,0) [above right]{${\scriptstyle n}$};
}\,\in W_{n^{+}+n^{-}}
$ 
inductively by the following. 
\begin{DEF}(The $A_2$ clasps)
\begin{align}
\tikz[baseline=-.6ex]{
\draw[->-=.8] (-.5,0) -- (.5,0);
\draw[fill=white] (-.1,-.3) rectangle (.1,.3);
\node at (.1,0) [above right]{${\scriptstyle 1}$};
}\,
&= 
\,\tikz[baseline=-.6ex]{
\draw[->-=.5] (-.5,0) -- (.5,0) node at (0,0) [above]{${\scriptstyle 1}$};
}\,\in W_{1^{+}+1^{-}}\notag\\
\tikz[baseline=-.6ex]{
\draw[->-=.8] (-.5,0) -- (.5,0);
\draw[fill=white] (-.1,-.3) rectangle (.1,.3);
\node at (.1,0) [above right]{${\scriptstyle n}$};
}\,
&=
\bigg\langle\,\tikz[baseline=-.6ex]{
\draw[->-=.8] (-.5,.1) -- (.5,.1);
\draw[->-=.5] (-.5,-.4) -- (.5,-.4);
\draw[fill=white] (-.1,-.2) rectangle (.1,.4);
\node at (0,0) [above right]{${\scriptstyle n-1}$};
\node at (.2,-.4) [above right]{${\scriptstyle 1}$};
}\,\bigg\rangle_{\! 3}
-\frac{\left[n-1\right]}{\left[n\right]}
\bigg\langle\,\tikz[baseline=-.6ex]{
\draw[->-=.5] (-.9,.1) -- (-.4,.1);
\draw[->-=.5] (-.3,.2) -- (.3,.2);
\draw[->-=.5] (.4,.1) -- (.9,.1);
\draw[->-=.5] (-.9,-.4) 
to[out=east, in=west] (-.3,-.4) 
to[out=east, in=south] (-.1,-.2);
\draw[-<-=.8] (-.1,-.2) 
to[out=north, in=east] (-.3,0);
\draw[-<-=.5] (.9,-.4) 
to[out=west, in=east] (.3,-.4)
to[out=west, in=south] (.1,-.2);
\draw[->-=.8] (.1,-.2)
to[out=north, in=west] (.3,0);
\draw[fill=white] (-.4,-.2) rectangle (-.3,.4);
\draw[fill=white] (.3,-.2) rectangle (.4,.4);
\draw[-<-=.5] (-.1,-.2) -- (.1,-.2);
\node at (-.3,0) [above left]{${\scriptscriptstyle n-1}$};
\node at (.3,0) [above right]{${\scriptscriptstyle n-1}$};
\node at (0,.1) [above]{${\scriptscriptstyle n-2}$};
\node at (-.6,-.3) {${\scriptscriptstyle 1}$};
\node at (.6,-.3) {${\scriptscriptstyle 1}$};
\node at (0,-.2) [below]{${\scriptscriptstyle 1}$};
\node at (-.3,0) [right]{${\scriptscriptstyle 1}$};
\node at (.3,0) [left]{${\scriptscriptstyle 1}$};
}\,\bigg\rangle_{\! 3} \in W_{n^{+}+n^{-}}
\end{align}
\end{DEF}
$A_2$ clasps have the following properties.
\begin{LEM}[Properties of $A_2$ clasps]
For any positive integer $n$,
\begin{itemize} 
\item $\Big\langle\,\tikz[baseline=-.6ex]{
\draw[->-=.5] (-.6,0) -- (.6,0);
\draw[fill=white] (-.4,-.3) rectangle (-.2,.3);
\draw[fill=white] (.2,-.3) rectangle (.4,.3);
\node at (-.3,0) [above right]{${\scriptstyle n}$};
\node at (.3,0) [above right]{${\scriptstyle n}$};
}\,\Big\rangle_{\! 3}
=
\,\tikz[baseline=-.6ex]{
\draw[->-=.8] (-.5,0) -- (.5,0);
\draw[fill=white] (-.1,-.3) rectangle (.1,.3);
\node at (0,0) [above right] {${\scriptstyle n}$};
}\,
$,
\item $\Big\langle\,\tikz[baseline=-.6ex]{
\draw[->-=.5] (-.5,0) -- (-.1,0);
\draw[->-=.5] (0,.2) -- (.5,.2);
\draw[->-=.5] (0,-.2) -- (.5,-.2);
\draw[->-=.5] (0,.1) 
to[out=east, in=north] (.3,0);
\draw[->-=.5] (0,-.1)
to[out=east, in=south] (.3,0);
\draw[-<-=.5] (.3,0) -- (.5,0);
\draw[fill=white] (-.2,-.3) rectangle (0,.3);
\node at (.4,0) [right] {${\scriptstyle 1}$};
\node at (.4,.2) [right] {${\scriptstyle n-k-2}$};
\node at (.4,-.2) [right] {${\scriptstyle k}$};
}\,\Big\rangle_{\! 3}=0$\quad ($k=0,1,\dots,n-2$).
\end{itemize}
\end{LEM}

We also define the $A_2$ clasp of type $(n,m)$ according to Ohtsuki and Yamada~\cite{OhtsukiYamada97}.
\begin{DEF}[the $A_2$ clasp of type $(n,m)$]\label{doubleA2clasp}
\[
\Bigg\langle\,\tikz[baseline=-.6ex, scale=.8]{
\draw[-<-=.8] (-.6,.4) -- (.6,.4);
\draw[->-=.8] (-.6,-.4) -- (.6,-.4);
\draw[fill=white] (-.1,-.6) rectangle (.1,.6);
\draw (-.1,.0) -- (.1,.0);
\node at (.4,-.6)[right]{$\scriptstyle{m}$};
\node at (-.4,-.6)[left]{$\scriptstyle{m}$};
\node at (.4,.6)[right]{$\scriptstyle{n}$};
\node at (-.4,.6)[left]{$\scriptstyle{n}$};
}\,\Bigg\rangle_{\! 3}
=
\sum_{k=0}^{\min\{m,n\}}
(-1)^k
\frac{{n\brack k}{m\brack k}}{{n+m+1\brack k}}
\Bigg\langle\,\tikz[baseline=-.6ex]{
\draw
(-.4,.4) -- +(-.2,0)
(.4,-.4) -- +(.2,0)
(-.4,-.4) -- +(-.2,0)
(.4,.4) -- +(.2,0);
\draw[-<-=.5] (-.4,.5) -- (.4,.5);
\draw[->-=.5] (-.4,-.5) -- (.4,-.5);
\draw[-<-=.5] (-.4,.3) to[out=east, in=east] (-.4,-.3);
\draw[->-=.5] (.4,.3) to[out=west, in=west] (.4,-.3);
\draw[fill=white] (.4,-.6) rectangle +(.1,.4);
\draw[fill=white] (-.4,-.6) rectangle +(-.1,.4);
\draw[fill=white] (.4,.6) rectangle +(.1,-.4);
\draw[fill=white] (-.4,.6) rectangle +(-.1,-.4);
\node at (.4,-.6)[right]{$\scriptstyle{m}$};
\node at (-.4,-.6)[left]{$\scriptstyle{m}$};
\node at (.4,.6)[right]{$\scriptstyle{n}$};
\node at (-.4,.6)[left]{$\scriptstyle{n}$};
\node at (0,.5)[above]{$\scriptstyle{n-k}$};
\node at (0,-.5)[below]{$\scriptstyle{m-k}$};
\node at (-.2,0)[left]{$\scriptstyle{k}$};
\node at (.2,0)[right]{$\scriptstyle{k}$};
}\,\Bigg\rangle_{\! 3}
\]
\end{DEF}

\begin{LEM}[Property of $A_2$ clasps of type $(m,n)$]\label{doubleA2claspprop}
\[
\Big\langle\,\tikz[baseline=-.6ex]{
\draw[-<-=.5] (-.5,.2) -- (-.1,.2);
\draw[->-=.5] (-.5,-.2) -- (-.1,-.2);
\draw[-<-=.5] (0,.2) -- (.5,.2);
\draw[->-=.5] (0,-.2) -- (.5,-.2);
\draw[-<=.5] (0,.1) 
to[out=east, in=north] (.3,0);
\draw (0,-.1)
to[out=east, in=south] (.3,0);
\draw[fill=white] (-.2,-.3) rectangle (0,.3);
\draw (-.2,.0) -- (.0,.0);
\node at (.4,0) [right] {${\scriptstyle 1}$};
\node at (.4,.2) [right] {${\scriptstyle n-1}$};
\node at (.4,-.2) [right] {${\scriptstyle m-1}$};
}\,\Big\rangle_{\! 3}=0 .
\]
\end{LEM}

We review some formulas for clasped $A_2$ web spaces in~\cite{Yuasa16}. 
We define a $q$-Pochhammer symbol as 
\[
 (q;q)_k=\prod_{l=1}^{k}(1-q^l).
\]
We abbreviate it as $(q)_k$.
A $q$-binomial coefficient is defined by
\[
 {n\choose k}_q=\frac{(q;q)_n}{(q;q)_k(q;q)_{n-k}}
\]
for $k\leq n$.
If $k>n$,
we define it by $0$.
We also define a $q$-multinomial coefficient as
\[
 {n\choose n_1,n_2,\dots,n_m}_q=\frac{(q)_n}{(q)_{n_1}(q)_{n_2}\cdots(q)_{n_m}},
\]
where $n_1, n_2,\dots, n_m$ are non-negative integers such that $n_1+n_2+\dots+n_m=n$.

\begin{THM}[The $m$ full twists formula~\cite{Yuasa16}]\label{A2mfull}
\begin{align*}
\Bigg\langle\,\tikz[baseline=-.6ex]{
\begin{scope}[xshift=-1cm]
\draw 
(-.5,.4) -- +(-.2,0)
(-.5,-.4) -- +(-.2,0);
\draw[->-=1, white, double=black, double distance=0.4pt, ultra thick] 
(-.5,-.4) to[out=east, in=west] (.0,.4);
\draw[-<-=1, white, double=black, double distance=0.4pt, ultra thick] 
(-.5,.4) to[out=east, in=west] (.0,-.4);
\draw[white, double=black, double distance=0.4pt, ultra thick] 
(0,-.4) to[out=east, in=west] (.5,.4);
\draw[white, double=black, double distance=0.4pt, ultra thick] 
(0,.4) to[out=east, in=west] (.5,-.4);
\draw[fill=white] (-.5,-.6) rectangle +(-.1,.4);
\draw[fill=white] (-.5,.6) rectangle +(-.1,-.4);
\node at (-.5,-.6)[left]{$\scriptstyle{n}$};
\node at (-.5,.6)[left]{$\scriptstyle{n}$};
\end{scope}
\node at (.0,.0){$\cdots$};
\node at (.0,-.4)[below]{$\scriptstyle{m\text{ full twists}}$};
\begin{scope}[xshift=1cm]
\draw
(.5,-.4) -- +(.2,0)
(.5,.4) -- +(.2,0);
\draw[->-=1, white, double=black, double distance=0.4pt, ultra thick] 
(-.5,-.4) to[out=east, in=west] (.0,.4);
\draw[-<-=1, white, double=black, double distance=0.4pt, ultra thick] 
(-.5,.4) to[out=east, in=west] (.0,-.4);
\draw[white, double=black, double distance=0.4pt, ultra thick] 
(0,-.4) to[out=east, in=west] (.5,.4);
\draw[white, double=black, double distance=0.4pt, ultra thick] 
(0,.4) to[out=east, in=west] (.5,-.4);
\draw[fill=white] (.5,-.6) rectangle +(.1,.4);
\draw[fill=white] (.5,.6) rectangle +(.1,-.4);
\node at (.5,-.6)[right]{$\scriptstyle{n}$};
\node at (.5,.6)[right]{$\scriptstyle{n}$};
\end{scope}
}\,\Bigg\rangle_{\! 3}
&=q^{-\frac{2m}{3}(n^2+3n)}
\sum_{0\leq k_m\leq \cdots\leq k_1\leq n}
q^{n-k_m}
q^{\sum_{i=1}^{m}(k_i^2+2k_i)}\\
&\qquad\times\frac{(q)_n}{(q)_{k_m}}
{n \choose k_1',k_2',\dots,k_m',k_m}_{q}
\Bigg\langle\,\tikz[baseline=-.6ex]{
\draw
(-.4,.4) -- +(-.2,0)
(.4,-.4) -- +(.2,0)
(-.4,-.4) -- +(-.2,0)
(.4,.4) -- +(.2,0);
\draw[-<-=.5] (-.4,.5) -- (.4,.5);
\draw[->-=.5] (-.4,-.5) -- (.4,-.5);
\draw[-<-=.5] (-.4,.3) to[out=east, in=east] (-.4,-.3);
\draw[->-=.5] (.4,.3) to[out=west, in=west] (.4,-.3);
\draw[fill=white] (.4,-.6) rectangle +(.1,.4);
\draw[fill=white] (-.4,-.6) rectangle +(-.1,.4);
\draw[fill=white] (.4,.6) rectangle +(.1,-.4);
\draw[fill=white] (-.4,.6) rectangle +(-.1,-.4);
\node at (.4,-.6)[right]{$\scriptstyle{n}$};
\node at (-.4,-.6)[left]{$\scriptstyle{n}$};
\node at (.4,.6)[right]{$\scriptstyle{n}$};
\node at (-.4,.6)[left]{$\scriptstyle{n}$};
\node at (0,.5)[above]{$\scriptstyle{k_m}$};
\node at (0,-.5)[below]{$\scriptstyle{k_m}$};
\node at (-.2,0)[left]{$\scriptstyle{n-k_m}$};
\node at (.2,0)[right]{$\scriptstyle{n-k_m}$};
}\,\Bigg\rangle_{\! 3},
\end{align*}
where $k_i, k_i'$ are integers such that $k_0=n$, $k_{i+1}'=k_i-k_{i+1}$ for $i=0,1,\dots,m-1$.
\end{THM}

\begin{THM}[The $A_2$ bracket bubble skein expansion formula~\cite{Yuasa16}]\label{A2bubble}
\[
\Bigg\langle\,\tikz[baseline=-.6ex, scale=0.8]{
\draw (-.4,.5) -- +(-.2,0);
\draw (.4,-.5) -- +(.2,0);
\draw (-.4,-.5) -- +(-.2,0);
\draw (.4,.5) -- +(.2,0);
\draw[-<-=.5] (-.4,.5) -- (0,.5);
\draw[-<-=.5] (0,.5) -- (.4,.5);
\draw[->-=.5] (-.4,-.5) -- (0,-.5);
\draw[->-=.5] (0,-.5) -- (.4,-.5);
\draw[-<-=.5] (.05,.3) to[out=east, in=east] (.05,.-.3);
\draw[->-=.5] (-.05,.3) to[out=west, in=west] (-.05,-.3);
\draw[fill=white] (-.4,.3) rectangle +(-.1,.3);
\draw[fill=white] (.4,.3) rectangle +(.1,.3);
\draw[fill=white] (-.4,-.3) rectangle +(-.1,-.3);
\draw[fill=white] (.4,-.3) rectangle +(.1,-.3);
\draw[fill=white] (-.05,.2) rectangle +(.1,.4);
\draw[fill=white] (-.05,-.2) rectangle +(.1,-.4);
\node at (.4,-.5)[below right]{$\scriptstyle{m-l}$};
\node at (-.4,-.5)[below left]{$\scriptstyle{m-k}$};
\node at (.4,.5)[above right]{$\scriptstyle{n-l}$};
\node at (-.4,.5)[above left]{$\scriptstyle{n-k}$};
\node at (-.2,0)[left]{$\scriptstyle{k}$};
\node at (.2,0)[right]{$\scriptstyle{l}$};
\node at (0,-.6)[below]{$\scriptstyle{m}$};
\node at (0,.6)[above]{$\scriptstyle{n}$};
}\,\Bigg\rangle_{\! 3}
=
\sum_{t=\max\{k, l\}}^{\min\{k+l, n, m\}}
\frac{{n\brack t}{m\brack t}{t\brack k}{t\brack l}{n+m-t+2\brack n+m-k-l+2}}{{n\brack k}{m\brack k}{n\brack l}{m\brack l}}
\Bigg\langle\,\tikz[baseline=-.6ex, scale=0.8]{
\draw (-.4,.4) -- +(-.2,0);
\draw (.4,-.4) -- +(.2,0);
\draw (-.4,-.4) -- +(-.2,0);
\draw (.4,.4) -- +(.2,0);
\draw[-<-=.5] (-.4,.5) -- (.4,.5);
\draw[->-=.5] (-.4,-.5)  -- (.4,-.5);
\draw[-<-=.5] (-.4,.3) to[out=east, in=east] (-.4,.-.3);
\draw[->-=.5] (.4,.3) to[out=west, in=west] (.4,-.3);
\draw[fill=white] (-.4,.2) rectangle +(-.1,.4);
\draw[fill=white] (.4,.2) rectangle +(.1,.4);
\draw[fill=white] (-.4,-.2) rectangle +(-.1,-.4);
\draw[fill=white] (.4,-.2) rectangle +(.1,-.4);
\node at (.4,-.6)[right]{$\scriptstyle{m-l}$};
\node at (-.4,-.6)[left]{$\scriptstyle{m-k}$};
\node at (.4,.6)[right]{$\scriptstyle{n-l}$};
\node at (-.4,.6)[left]{$\scriptstyle{n-k}$};
\node at (-.2,0)[left]{$\scriptstyle{t-k}$};
\node at (.2,0)[right]{$\scriptstyle{t-l}$};
\node at (0,-.5)[below]{$\scriptstyle{m-t}$};
\node at (0,.5)[above]{$\scriptstyle{n-t}$};
}\,\Bigg\rangle_{\! 3},
\]
where 
${a \brack b}=\frac{\left[a\right]!}{\left[b\right]!\left[a-b\right]!}$
for $a\geq b$ and  $\left[a\right]!=\prod_{i=1}^a\left[i\right]$. 
\end{THM}

\section{Trivalent graphs for clasped $A_2$ web spaces}
For the Kauffman bracket skein module, 
a method for representing its elements by colored trivalent graphs is well known. 
Formulas related to colored trivalent graphs play a very important role in defining the quantum $SU(2)$ invariants for $3$-manifolds and the quantum representations for mapping class groups of surfaces, etc.
(See, for example, Kauffman and Lins~\cite{KauffmanLins94}, Turaev and Viro~\cite{TuraevViro92} and Roberts~\cite{Roberts94}.)
In this section, 
we represent a certain clasped $A_2$ web space by using colored trivalent graphs and give some formulas.

\begin{DEF}
Let $n$ be a non-negative integer. 
For $0\leq i\leq n$, 
we define a colored trivalent graph
$
\tikz[baseline=-.6ex, scale=0.5]{
\draw[triple={[line width=1.4pt, white] in [line width=2.2pt, black] in [line width=5.4pt, white]}]
(0,0) -- (1,0);
\draw[-<-=.5] (-1,1) to[out=east, in=north west] (0,0);
\draw[->-=.5] (-1,-1) to[out=east, in=south west] (0,0);
\node at (-1,1) [above]{$\scriptstyle{n}$};
\node at (-1,-1) [above]{$\scriptstyle{n}$};
\node at (.5,0) [above]{$\scriptstyle{i}$};
}
\text{\ by\ }
\tikz[baseline=-.6ex, scale=0.5]{
\draw (-1,.9) -- +(-.5,0);
\draw (-1,-.9) -- +(-.5,0);
\draw (1,.3) -- +(.5,0);
\draw (1,-.3) -- +(.5,0);
\draw[-<-=.5] (-1,1) to[out=east, in=west] (1,.3);
\draw[->-=.5] (-1,-1) to[out=east, in=west] (1,-.3);
\draw[-<-=.5] (-1,.8) to[out=east, in=east] (-1,.-.8);
\draw[fill=white] (-1.2,.6) rectangle (-1,1.2);
\draw[fill=white] (-1.2,-.6) rectangle (-1,-1.2);
\draw[fill=white] (1,-.6) rectangle (1.2,.6);
\draw (1,.0) -- (1.2,.0);
\node at (-1,.6)[left]{$\scriptstyle{n}$};
\node at (-1,-.6)[left]{$\scriptstyle{n}$};
\node at (-1.2,0){$\scriptstyle{n-i}$};
\node at (0,1){$\scriptstyle{i}$};
\node at (0,-1){$\scriptstyle{i}$};
}
$\ .
\end{DEF}
We use the following notations:
\begin{itemize}
\item $\Delta(m,n)=
\bigg\langle\,
\tikz[baseline=-.6ex, scale=0.5]{
\draw[-<-=.5] (0,0) circle (.6);
\draw[->-=.5] (0,0) circle (1);
\draw[fill=white] (.4,-.1) rectangle (1.2,.1);
\draw (.8,-.1) -- (.8,.1);
\node at (-.6,0) [right]{$\scriptstyle{n}$};
\node at (-1,0) [left]{$\scriptstyle{m}$};
}\,\bigg\rangle_{\! 3},
$
\item $\theta(n,n,(i,i))=
\bigg\langle\,
\tikz[baseline=-.6ex, scale=0.5]{
\draw[triple={[line width=1.4pt, white] in [line width=2.2pt, black] in [line width=5.4pt, white]}]
(0,0) -- (1,0);
\draw[->-=.5] 
(0,0) to[out=north, in=west] 
(.5,.8) to[out=east, in=north] (1,0);
\draw[-<-=.5] 
(0,0) to[out=south, in=west] 
(.5,-.8) to[out=east, in=south] (1,0);
\node at (0,.8) {$\scriptstyle{n}$};
\node at (0,-.8) {$\scriptstyle{n}$};
\node at (.5,.0) [above]{$\scriptstyle{i}$};
}\,\bigg\rangle_{\! 3},
$
\item $\operatorname{Tet}\!
\begin{bmatrix}
n&n&(j,j)\\
n&n&(i,i)
\end{bmatrix}=
\bigg\langle\,
\tikz[baseline=-.6ex, scale=0.5]{
\draw[triple={[line width=1.4pt, white] in [line width=2.2pt, black] in [line width=5.4pt, white]}]
(0,0) -- (1.2,0);
\draw[triple={[line width=1.4pt, white] in [line width=2.2pt, black] in [line width=5.4pt, white]}]
(.6,1) to[out=north, in=north] 
(2,1) -- (2,-1)
to[out=south, in=south] (.6,-1);
\draw[->-=.5] (0,0) -- (.6,1);
\draw[->-=.5] (.6,1) -- (1.2,0);
\draw[-<-=.5] (0,0) -- (.6,-1);
\draw[-<-=.5] (.6,-1) -- (1.2,0);
\node at (0,.8) {$\scriptstyle{n}$};
\node at (0,-.8) {$\scriptstyle{n}$};
\node at (1,.8) {$\scriptstyle{n}$};
\node at (1,-.8) {$\scriptstyle{n}$};
\node at (.6,.0) [above]{$\scriptstyle{i}$};
\node at (2,.0) [left]{$\scriptstyle{j}$};
}\,\bigg\rangle_{\! 3},
$
\item 
$\displaystyle
\begin{Bmatrix}
n&n&(j,j)\\
n&n&(i,i)
\end{Bmatrix}
=
\frac{\operatorname{Tet}\!
\begin{bmatrix}
n&n&(j,j)\\
n&n&(i,i)
\end{bmatrix}\Delta(j,j)}{\theta(n,n,(j,j))^2},
$
\end{itemize}
where $m,n$ are any non-negative integers and $0\leq i,j\leq n$.

\begin{LEM}\ 
\begin{enumerate}
\item $\Delta(i,j)=\frac{\left[i+1\right]\left[j+1\right]\left[i+j+2\right]}{\left[2\right]}$,
\item $\theta(n,n,(i,i))=\sum_{k=0}^i(-1)^k\frac{{i\brack k}^2}{{2i+1\brack k}}\frac{\Delta(n,0)^2}{\Delta(n-i+k,0)}=\frac{{n+i+2\brack 2i+2}}{{n\brack i}^2}\Delta(i,i)$,
\item $\operatorname{Tet}\!
\begin{bmatrix}
n&n&(j,j)\\
n&n&(i,i)
\end{bmatrix}
=\sum_{k=0}^i(-1)^k\frac{{i\brack k}^2{n-j\brack i-k}{n+j+2\brack i-k}}{{2i+1\brack k}{n\brack i-k}^2}\theta(n,n,(j,j))$.
\end{enumerate}
\end{LEM}

\begin{proof}
We only show $(2)$ and $(3)$.
The first equation of $(2)$ is obtained by expanding the double-lined edge using Definition~\ref{doubleA2clasp}.
$\theta(n,n,(i,i))$ is also represented by a colored graph
$\bigg\langle\,
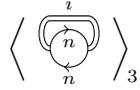
\tikz[baseline=-.6ex, scale=0.5]{
\draw[triple={[line width=1.4pt, white] in [line width=2.2pt, black] in [line width=5.4pt, white]}]
(0,0) to[out=west, in=west] (0,.8) -- (1,.8) to[out=east, in=east] (1,0);
\draw[->-=.5] 
(0,0) to[out=north, in=west] 
(.5,.5) to[out=east, in=north] (1,0);
\draw[-<-=.5] 
(0,0) to[out=south, in=west] 
(.5,-.5) to[out=east, in=south] (1,0);
\node at (.5,.5) [below]{$\scriptstyle{n}$};
\node at (.5,-.5) [below]{$\scriptstyle{n}$};
\node at (.5,.8) [above]{$\scriptstyle{i}$};
}\,\bigg\rangle_{\! 3}$. 
This clasped $A_2$ web has a bubble skein element colored by $n-i$ and we use Theorem~\ref{A2bubble}.
Thus, 
we obtain the second equation of $(2)$.
The proof of $(3)$ is obtained by using Definition~\ref{doubleA2clasp} and Theorem~\ref{A2bubble}. 
We expand a double-lined edge colored by $s$.
Then,
the clasped $A_2$ web appearing in each term has two bubble skein elements. 
Finally, $\theta(n,n,(j,j))$ is obtained by expanding one of these bubble skein elements.
\end{proof}

\begin{LEM}\label{bigongraph}
\[
\bigg\langle\,
\tikz[baseline=-.6ex, scale=0.5]{
\draw[triple={[line width=1.4pt, white] in [line width=2.2pt, black] in [line width=5.4pt, white]}]
(-1,0)--(-.2,0);
\draw[triple={[line width=1.4pt, white] in [line width=2.2pt, black] in [line width=5.4pt, white]}]
(1.2,0)--(2,0);
\draw[-<-=.5] 
(-.2,0) to[out=north, in=west] 
(.6,.5) to[out=east, in=north] (1.2,0);
\draw[->-=.5] 
(-.2,0) to[out=south, in=west] 
(.6,-.5) to[out=east, in=south] (1.2,0);
\node at (.6,.5) [above]{$\scriptstyle{n}$};
\node at (.6,-.5) [below]{$\scriptstyle{n}$};
\node at (-1,0) [above]{$\scriptstyle{i}$};
\node at (2,0) [above]{$\scriptstyle{j}$};
}\,\bigg\rangle_{\! 3}
=\delta_{ij}
\frac{\theta(n,n,(i,i))}{\Delta(i,i)}
\bigg\langle
\,\tikz[baseline=-.6ex, scale=0.5]{
\draw[triple={[line width=1.4pt, white] in [line width=2.2pt, black] in [line width=5.4pt, white]}]
(-1,0)--(1,0);
\node at (0,0) [above]{$\scriptstyle{i}$};
}\,\bigg\rangle_{\! 3},
\]
where 
\,\tikz[baseline=-.6ex, scale=0.5]{
\draw[triple={[line width=1.4pt, white] in [line width=2.2pt, black] in [line width=5.4pt, white]}]
(-1,0)--(1,0);
\node at (0,0) [above]{$\scriptstyle{i}$};
}\, 
denotes 
\,\tikz[baseline=-.6ex, scale=.5]{
\draw[-<-=.8] (-.6,.4) -- (.6,.4);
\draw[->-=.8] (-.6,-.4) -- (.6,-.4);
\draw[fill=white] (-.1,-.6) rectangle (.1,.6);
\draw (-.1,.0) -- (.1,.0);
\node at (.4,-.6)[right]{$\scriptstyle{i}$};
\node at (-.4,-.6)[left]{$\scriptstyle{i}$};
\node at (.4,.6)[right]{$\scriptstyle{i}$};
\node at (-.4,.6)[left]{$\scriptstyle{i}$};
}\,
and $\delta_{ij}$ is 
the Kronecker delta function.
\end{LEM}

\begin{proof}
Theorem~\ref{A2bubble} and Lemma~\ref{doubleA2claspprop} show the value is zero if $i\neq j$.
If $i=j$, 
then the coefficient is obtained by closing double-lined edges of both sides.
\end{proof}

\begin{PROP}[Recoupling Theorem]\label{recoupling}
\[
\bigg\langle\tikz[baseline=-.6ex, scale=0.5]{
\draw[triple={[line width=1.4pt, white] in [line width=2.2pt, black] in [line width=5.4pt, white]}]
(0,0) -- (1,0);
\draw[-<-=.5] (-1,1) -- (0,0);
\draw[->-=.5] (-1,-1) -- (0,0);
\draw[-<-=.5] (1,0) -- (2,1);
\draw[->-=.5] (1,0) -- (2,-1);
\node at (-1,1) [below]{$\scriptstyle{n}$};
\node at (-1,-1) [above]{$\scriptstyle{n}$};
\node at (2,1) [below]{$\scriptstyle{n}$};
\node at (2,-1) [above]{$\scriptstyle{n}$};
\node at (.5,0) [above]{$\scriptstyle{i}$};
}\,\bigg\rangle_{\! 3}
=\sum_{j=0}^{n}
\begin{Bmatrix}
n&n&(j,j)\\
n&n&(i,i)
\end{Bmatrix}
\bigg\langle\tikz[baseline=-.6ex, scale=0.5]{
\draw[triple={[line width=1.4pt, white] in [line width=2.2pt, black] in [line width=5.4pt, white]}]
(0,-.5) -- (0,.5);
\draw[-<-=.5] (-1,1) -- (0,.5);
\draw[->-=.5] (-1,-1) -- (0,-.5);
\draw[-<-=.5] (0,.5) -- (1,1);
\draw[->-=.5] (0,-.5) -- (1,-1);
\node at (-1,1) [below]{$\scriptstyle{n}$};
\node at (-1,-1) [above]{$\scriptstyle{n}$};
\node at (1,1) [below]{$\scriptstyle{n}$};
\node at (1,-1) [above]{$\scriptstyle{n}$};
\node at (0,0) [right]{$\scriptstyle{j}$};
}\,\bigg\rangle_{\!3}
\]
\end{PROP}

\begin{proof}
We can prove this proposition in the same way as the recoupling theory of Temperley-Lieb algebra, see \cite[Chapter~7]{KauffmanLins94}, using Definition~\ref{doubleA2clasp} and Lemma~\ref{bigongraph}.
\end{proof}

We give a formula of the $\mathfrak{sl}_3$ colored Jones polynomial $J_{(n,0)}^{\mathfrak{sl}_3}(\left[2a_1,2a_2,\dots,2a_l\right];q)$ for a $2$-bridge link $\left[2a_1,2a_2,\dots,2a_l\right]$ using colored trivalent graphs. 
In this paper we define the $\mathfrak{sl}_3$ colored Jones polynomial for framed links obtained by a link diagram with the blackboard framing.
\begin{DEF}
\begin{align*}
&J_{(n,0)}^{\mathfrak{sl}_3}(\left[2a_1,2a_2,\dots,2a_l\right];q)\\
&\quad=
\begin{cases}
\bigg\langle\tikz[baseline=-.6ex, scale=1.2]{
\begin{scope}
\draw (-.3,.3) -- (-.5,.3) to[out=west, in=west] (-.5,.1) -- (-.3,.1);
\draw (-.3,-.1) -- (-.5,-.1) to[out=west, in=west] (-.5,-.3) -- (-.3,-.3);
\draw[fill=white] (-.7,.15) rectangle (-.4,.25);
\node at (-.6,.0) {$\scriptstyle{n}$};
\draw[fill=white] (-.7,-.15) rectangle (-.4,-.25);
\node at (-.6,-.4) {$\scriptstyle{n}$};
\draw[->-=.5] (-.3,.3) -- (.3,.3);
\draw (-.3,.1) -- (.3,.1);
\draw (-.3,-.1) -- (.3,-.1);
\draw[-<-=.5] (-.3,-.3) -- (.3,-.3);
\draw[fill=white] (-.2,-.2) rectangle (.2,.2);
\node at (.0,.0) {$\scriptstyle{2a_1}$};
\end{scope}
\begin{scope}[xshift=.6cm]
\draw (-.3,.3) -- (.3,.3);
\draw (-.3,.1) -- (.3,.1);
\draw (-.3,-.1) -- (.3,-.1);
\draw (-.3,-.3) -- (.3,-.3);
\draw[fill=white] (-.2,-.4) rectangle (.2,.0);
\node at (.0,-.2) {$\scriptstyle{2a_2}$};
\end{scope}
\begin{scope}[xshift=1.2cm]
\draw (-.3,.3) -- (.3,.3);
\draw (-.3,.1) -- (.3,.1);
\draw (-.3,-.1) -- (.3,-.1);
\draw (-.3,-.3) -- (.3,-.3);
\draw[fill=white] (-.2,-.2) rectangle (.2,.2);
\node at (.0,.0) {$\scriptstyle{2a_3}$};
\end{scope}
\begin{scope}[xshift=1.8cm]
\draw (-.3,.3) -- (.3,.3);
\draw (-.3,.1) -- (.3,.1);
\draw (-.3,-.1) -- (.3,-.1);
\draw (-.3,-.3) -- (.3,-.3);
\draw[fill=white] (-.2,-.4) rectangle (.2,.0);
\node at (.0,-.2) {$\scriptstyle{2a_4}$};
\end{scope}
\begin{scope}[xshift=2.4cm]
\draw (-.3,.3) -- (.3,.3);
\node at (.0,.0) {$\cdots$};
\node at (.0,-.2) {$\cdots$};
\end{scope}
\begin{scope}[xshift=3cm]
\draw (.3,.3) to[out=east, in=east] (.3,.1);
\draw (.3,-.1) to[out=east, in=east] (.3,-.3);
\draw (-.3,.3) -- (.3,.3);
\draw (-.3,.1) -- (.3,.1);
\draw (-.3,-.1) -- (.3,-.1);
\draw (-.3,-.3) -- (.3,-.3);
\draw[fill=white] (-.2,-.2) rectangle (.2,.2);
\node at (.0,.0) {$\scriptstyle{2a_l}$};
\end{scope}
}\bigg\rangle_{\!3} \Big/
\Delta(n,0)
&\ \text{if $l$ is odd,}\\
\bigg\langle\tikz[baseline=-.6ex, scale=1.2]{
\begin{scope}
\draw (-.3,.3) -- (-.5,.3) to[out=west, in=west] (-.5,.1) -- (-.3,.1);
\draw (-.3,-.1) -- (-.5,-.1) to[out=west, in=west] (-.5,-.3) -- (-.3,-.3);
\draw[fill=white] (-.7,.15) rectangle (-.4,.25);
\node at (-.6,.0) {$\scriptstyle{n}$};
\draw[fill=white] (-.7,-.15) rectangle (-.4,-.25);
\node at (-.6,-.4) {$\scriptstyle{n}$};
\draw[->-=.5] (-.3,.3) -- (.3,.3);
\draw (-.3,.1) -- (.3,.1);
\draw (-.3,-.1) -- (.3,-.1);
\draw[-<-=.5] (-.3,-.3) -- (.3,-.3);
\draw[fill=white] (-.2,-.2) rectangle (.2,.2);
\node at (.0,.0) {$\scriptstyle{2a_1}$};
\end{scope}
\begin{scope}[xshift=.6cm]
\draw (-.3,.3) -- (.3,.3);
\draw (-.3,.1) -- (.3,.1);
\draw (-.3,-.1) -- (.3,-.1);
\draw (-.3,-.3) -- (.3,-.3);
\draw[fill=white] (-.2,-.4) rectangle (.2,.0);
\node at (.0,-.2) {$\scriptstyle{2a_2}$};
\end{scope}
\begin{scope}[xshift=1.2cm]
\draw (-.3,.3) -- (.3,.3);
\draw (-.3,.1) -- (.3,.1);
\draw (-.3,-.1) -- (.3,-.1);
\draw (-.3,-.3) -- (.3,-.3);
\draw[fill=white] (-.2,-.2) rectangle (.2,.2);
\node at (.0,.0) {$\scriptstyle{2a_3}$};
\end{scope}
\begin{scope}[xshift=1.8cm]
\draw (-.3,.3) -- (.3,.3);
\draw (-.3,.1) -- (.3,.1);
\draw (-.3,-.1) -- (.3,-.1);
\draw (-.3,-.3) -- (.3,-.3);
\draw[fill=white] (-.2,-.4) rectangle (.2,.0);
\node at (.0,-.2) {$\scriptstyle{2a_4}$};
\end{scope}
\begin{scope}[xshift=2.4cm]
\draw (-.3,.3) -- (.3,.3);
\node at (.0,.0) {$\cdots$};
\node at (.0,-.2) {$\cdots$};
\end{scope}
\begin{scope}[xshift=3cm]
\draw (.3,.3) to[out=east, in=east] (.3,-.3);
\draw (.3,.1) to[out=east, in=east] (.3,-.1);
\draw (-.3,.3) -- (.3,.3);
\draw (-.3,.1) -- (.3,.1);
\draw (-.3,-.1) -- (.3,-.1);
\draw (-.3,-.3) -- (.3,-.3);
\draw[fill=white] (-.2,-.4) rectangle (.2,.0);
\node at (.0,-.2) {$\scriptstyle{2a_l}$};
\end{scope}
}\bigg\rangle_{\!3} \Big/
\Delta(n,0)
&\ \text{if $l$ is even,}
\end{cases}
\end{align*}
where, $a_1,a_2,\dots,a_l$ are non-zero integers and 
\begin{align*} 
\tikz[baseline=-.6ex]{
\draw (-.5,.2) -- (.5,.2);
\draw (-.5,-.2) -- (.5,-.2);
\draw[fill=white] (-.2,-.3) rectangle (.2,.3);
\node at (.0,.0) {${\scriptstyle m}$};
}\,
=
\begin{cases}
\,\tikz[baseline=-.6ex]{
\begin{scope}[xshift=-.5cm]
\draw (-.5,.2) -- +(-.2,0);
\draw (-.5,-.2) -- +(-.2,0);
\draw[white, double=black, double distance=0.4pt, ultra thick] 
(-.5,-.2) to[out=east, in=west] (.0,.2);
\draw[white, double=black, double distance=0.4pt, ultra thick] 
(-.5,.2) to[out=east, in=west] (.0,-.2);
\end{scope}
\node at (.0,.0){$\cdots$};
\node at (.0,-.2)[below]{$\scriptstyle{\text{right-handed }m\text{ half twists}}$};
\begin{scope}[xshift=.5cm]
\draw
(.5,-.2) -- +(.2,0)
(.5,.2) -- +(.2,0);
\draw[white, double=black, double distance=0.4pt, ultra thick] 
(0,-.2) to[out=east, in=west] (.5,.2);
\draw[white, double=black, double distance=0.4pt, ultra thick] 
(0,.2) to[out=east, in=west] (.5,-.2);
\end{scope}
}\,  & \text{if $m>$0,}\\
\,\tikz[baseline=-.6ex]{
\begin{scope}[xshift=-.5cm]
\draw (-.5,.2) -- +(-.2,0);
\draw (-.5,-.2) -- +(-.2,0);
\draw[white, double=black, double distance=0.4pt, ultra thick] 
(-.5,.2) to[out=east, in=west] (.0,-.2);
\draw[white, double=black, double distance=0.4pt, ultra thick] 
(-.5,-.2) to[out=east, in=west] (.0,.2);
\end{scope}
\node at (.0,.0){$\cdots$};
\node at (.0,-.2)[below]{$\scriptstyle{\text{left-handed } m\text{ half twists}}$};
\begin{scope}[xshift=.5cm]
\draw
(.5,-.2) -- +(.2,0)
(.5,.2) -- +(.2,0);
\draw[white, double=black, double distance=0.4pt, ultra thick] 
(0,.2) to[out=east, in=west] (.5,-.2);
\draw[white, double=black, double distance=0.4pt, ultra thick] 
(0,-.2) to[out=east, in=west] (.5,.2);
\end{scope}
}\, & \text{if $m<0$.}
\end{cases}
\end{align*} 
\end{DEF}
\begin{THM}\label{coloredjones}
\begin{align*}
&J_{(n,0)}^{\mathfrak{sl}_3}(\left[2a_1,2a_2,\dots,2a_l\right];q)\\
&\quad=\sum_{0\leq i_1, i_2\dots i_l\leq n}
\frac{\Delta(i_1,i_1)}{\Delta(n,0)}\frac{\theta(n,n,(i_l,i_l))}{\theta(n,n,(i_1,i_1))}
q^{-\frac{2}{3}(n^2+3n)(\sum_{k=1}^l a_k)}q^{\sum_{k=1}^l a_k(i_k^2+2i_k)}\\
&\qquad\times \prod_{k=1}^{l-1}
\begin{Bmatrix}
n & n & (i_{k+1},i_{k+1})\\
n & n & (i_{k},i_{k})
\end{Bmatrix}.
\end{align*}
\end{THM}

\begin{LEM}\label{twistcoeff}
\[
\bigg\langle\,\tikz[baseline=-.6ex]{
\draw[triple={[line width=1.4pt, white] in [line width=2.2pt, black] in [line width=5.4pt, white]}]
(-1.5,0) -- (-1,0);
\draw (-1,0) to[out=north east, in=west] (-.5,.4);
\draw (-1,0) to[out=south east, in=west] (-.5,-.4);
\draw[->-=1, white, double=black, double distance=0.4pt, ultra thick] 
(-.5,-.4) to[out=east, in=west] (.0,.4);
\draw[-<-=1, white, double=black, double distance=0.4pt, ultra thick] 
(-.5,.4) to[out=east, in=west] (.0,-.4);
\draw[white, double=black, double distance=0.4pt, ultra thick] 
(0,-.4) to[out=east, in=west] (.5,.4);
\draw[white, double=black, double distance=0.4pt, ultra thick] 
(0,.4) to[out=east, in=west] (.5,-.4);
\node at (-.5,-.5) [left]{$\scriptstyle{n}$};
\node at (-.5,.5) [left]{$\scriptstyle{n}$};
\node at (-1.5,0) [above]{$\scriptstyle{i}$};
}\,\bigg\rangle_{\!3}
=
q^{-\frac{2}{3}(n^2+3n)+i^2+2i}
\bigg\langle\,\tikz[baseline=-.6ex]{
\draw[triple={[line width=1.4pt, white] in [line width=2.2pt, black] in [line width=5.4pt, white]}]
(-1.5,0) -- (-1,0);
\draw[-<-=.5] (-1,0) to[out=north east, in=west] (-.5,.4);
\draw[->-=.5] (-1,0) to[out=south east, in=west] (-.5,-.4);
\node at (-.5,-.5) [left]{$\scriptstyle{n}$};
\node at (-.5,.5) [left]{$\scriptstyle{n}$};
\node at (-1.5,0) [above]{$\scriptstyle{i}$};
}\,\bigg\rangle_{\!3}
\]
\end{LEM}
\begin{proof}
It is shown by using the following formula~\cite[Lemma~{3.16}]{Yuasa16}:
\[
\Bigg\langle\,\tikz[baseline=-.6ex, rotate=180]{
\draw 
(-.5,.4) -- +(-.2,0)
(-.5,-.4) -- +(-.2,0);
\draw[->-=.7] (.0,-.5) -- +(.4,0);
\draw[-<-=.7] (.0,.5) -- +(.4,0);
\draw[-<-=.2, white, double=black, double distance=0.4pt, ultra thick] 
(-.5,-.4) to[out=east, in=west] (.0,.4);
\draw[->-=.2, white, double=black, double distance=0.4pt, ultra thick] 
(-.5,.4) to[out=east, in=west] (.0,-.4);
\draw[-<-=.5] (.1,.4) to[out=east, in=east] (.1,-.4);
\draw[fill=white] (.0,-.6) rectangle +(.1,.4);
\draw[fill=white] (-.5,-.6) rectangle +(-.1,.4);
\draw[fill=white] (.0,.6) rectangle +(.1,-.4);
\draw[fill=white] (-.5,.6) rectangle +(-.1,-.4);
\node at (-.4,-.6){$\scriptstyle{n}$};
\node at (.5,-.6){$\scriptstyle{i}$};
\node at (-.4,.6){$\scriptstyle{n}$};
\node at (.5,.6){$\scriptstyle{i}$};
\node at (.3,0)[left]{$\scriptstyle{n-i}$};
}\,\Bigg\rangle_{\! 3}
=q^{-\frac{n^2+3n-i^2-3i}{3}}
\Bigg\langle\,\tikz[baseline=-.6ex, rotate=180]{
\draw (-.4,.4) -- +(-.2,0);
\draw (.4,-.4) -- +(.2,0);
\draw (-.4,-.4) -- +(-.2,0);
\draw (.4,.4) -- +(.2,0);
\draw[-<-=.3, white, double=black, double distance=0.4pt, ultra thick] 
(-.4,-.5) to[out=east, in=west] (.4,.4);
\draw[->-=.3, white, double=black, double distance=0.4pt, ultra thick] 
(-.4,.5) to[out=east, in=west] (.4,-.4);
\draw[->-=.5] (-.4,.3) to[out=east, in=east] (-.4,-.3);
\draw[fill=white] (.4,-.6) rectangle +(.1,.4);
\draw[fill=white] (-.4,-.6) rectangle +(-.1,.4);
\draw[fill=white] (.4,.6) rectangle +(.1,-.4);
\draw[fill=white] (-.4,.6) rectangle +(-.1,-.4);
\node at (.4,-.6)[right]{$\scriptstyle{i}$};
\node at (.4,.6)[right]{$\scriptstyle{i}$};
\node at (-.4,0)[right]{$\scriptstyle{n-i}$};
}\,\Bigg\rangle_{\! 3},
\]
and Theorem~\ref{A2mfull} with $m=1$.
\end{proof}

\begin{proof}[Proof of Theorem~\ref{coloredjones}]
\begin{align}
\bigg\langle\tikz[baseline=-.6ex, scale=.5]{
\draw[-<-=.2] (-2,.5) -- (1,.5);
\draw[->-=.2] (-2,-.5) -- (1,-.5);
\draw[fill=white] (-1,.2) rectangle (-.8,.8);
\draw[fill=white] (-1,-.2) rectangle (-.8,-.8);
\draw[fill=white] (-.5,-1) rectangle (.5,1);
\node at (-1,.5) [above left]{${\scriptstyle n}$};
\node at (-1,-.5) [below left]{${\scriptstyle n}$};
\node at (.0,.0) {${\scriptstyle 2a_k}$};
}\,\bigg\rangle_{\!3}
&=
\bigg\langle\tikz[baseline=-.6ex, scale=.5]{
\draw[triple={[line width=1.4pt, white] in [line width=2.2pt, black] in [line width=5.4pt, white]}]
(-1.2,-.5) -- (-1.2,.5);
\draw[-<-=.1, -<-=.4] (-2,.5) -- (1,.5);
\draw[->-=.1, ->-=.4] (-2,-.5) -- (1,-.5);
\draw[fill=white] (-.5,-1) rectangle (.5,1);
\node at (-1.2,0) [left]{${\scriptstyle 0}$};
\node at (-1.2,.5) [above left]{${\scriptstyle n}$};
\node at (-1.2,-.5) [below left]{${\scriptstyle n}$};
\node at (-1.2,.5) [above right]{${\scriptstyle n}$};
\node at (-1.2,-.5) [below right]{${\scriptstyle n}$};
\node at (.0,.0) {${\scriptstyle 2a_k}$};
}\,\bigg\rangle_{\!3}
=
\sum_{j=0}^{n}
\begin{Bmatrix}
n&n&(i_k,i_k)\\
n&n&(0,0)
\end{Bmatrix}
\bigg\langle\tikz[baseline=-.6ex, scale=.5]{
\draw[triple={[line width=1.4pt, white] in [line width=2.2pt, black] in [line width=5.4pt, white]}]
(-1.5,0) -- (-1,0);
\draw[-<-=.5] (-2,.5) to[out=east, in=north west] (-1.5,0);
\draw[->-=.5] (-2,-.5) to[out=east, in=south west] (-1.5,0);
\draw[-<-=.5] (-1,0) to[out=north east, in=west] (-.5,.5);
\draw[->-=.5] (-1,0) to[out=south east, in=west] (-.5,-.5);
\draw[-<-=.1, -<-=.4] (0,.5) -- (1,.5);
\draw[->-=.1, ->-=.4] (0,-.5) -- (1,-.5);
\draw[fill=white] (-.5,-1) rectangle (.5,1);
\node at (-1.25,0) [above]{${\scriptstyle i_k}$};
\node at (-1.2,.5) [above left]{${\scriptstyle n}$};
\node at (-1.2,-.5) [below left]{${\scriptstyle n}$};
\node at (-1.2,.5) [above right]{${\scriptstyle n}$};
\node at (-1.2,-.5) [below right]{${\scriptstyle n}$};
\node at (.0,.0) {${\scriptstyle 2a_k}$};
}\,\bigg\rangle_{\!3}\label{eq1}\\
&=q^{-\frac{2a_k}{3}(n^2+3n)+a_k(i_k^2+2i_k)}\frac{\Delta(i_k,i_k)}{\theta(n,n,(i_k,i_k))}
\bigg\langle\tikz[baseline=-.6ex, scale=.5]{
\draw[triple={[line width=1.4pt, white] in [line width=2.2pt, black] in [line width=5.4pt, white]}]
(-.5,0) -- (.5,0);
\draw[-<-=.5] (-1,.5) to[out=east, in=north west] (-.5,0);
\draw[->-=.5] (-1,-.5) to[out=east, in=south west] (-.5,0);
\draw[-<-=.5] (.5,0) to[out=north east, in=west] (1,.5);
\draw[->-=.5] (.5,0) to[out=south east, in=west] (1,-.5);
\node at (0,0) [above]{${\scriptstyle i_k}$};
\node at (1,.5) [above]{${\scriptstyle n}$};
\node at (1,-.5) [below]{${\scriptstyle n}$};
\node at (-1,.5) [above]{${\scriptstyle n}$};
\node at (-1,-.5) [below]{${\scriptstyle n}$};
}\,\bigg\rangle_{\!3}.\notag
\end{align}
This equation is derived from Proposition~\ref{recoupling}, Lemma~\ref{twistcoeff} and 
$\begin{Bmatrix}
n&n&(i_k,i_k)\\
n&n&(0,0)
\end{Bmatrix}
=\frac{\Delta(i_k,i_k)}{\theta(n,n,(i_k,i_k))}
$.
By Proposition~\ref{recoupling} and Lemma~\ref{bigongraph},
\begin{align}
\bigg\langle\tikz[baseline=-.6ex, scale=.5]{
\draw[triple={[line width=1.4pt, white] in [line width=2.2pt, black] in [line width=5.4pt, white]}]
(0,0) -- (1,0);
\draw[triple={[line width=1.4pt, white] in [line width=2.2pt, black] in [line width=5.4pt, white]}]
(2,-.6) -- (3,-.6);
\draw[->-=.5] (0,0) to[out=west, in=west] (0,.9) -- (1.5,.9);
\draw[-<-=.5] (0,0) to[out=west, in=west] (0,-.9) -- (1.5,-.9);
\draw[-<-=.8] (1,0) to[out=north east, in=west] (1.5,.3);
\draw[->-=.8] (1,0) to[out=south east, in=west] (1.5,-.3);
\draw (1.5,.9) -- (3.5,.9);
\draw (1.5,.3) -- (3.5,.3);
\draw (1.5,-.3) to[out=east, in=north west] (2,-.6);
\draw (1.5,-.9) to[out=east, in=south west] (2,-.6);
\draw[->-=.8] (3,-.6) to[out=north east, in=west] (3.5,-.3);
\draw[-<-=.8] (3,-.6) to[out=south east, in=west] (3.5,-.9);
\node at (.5,0) [above]{${\scriptstyle i_k}$};
\node at (2.5,-.6) [above]{${\scriptstyle i_{k+1}}$};
\node at (-.5,0) [above]{${\scriptstyle n}$};
\node at (-.5,0) [below]{${\scriptstyle n}$};
\node at (1,0) [above]{${\scriptstyle n}$};
\node at (1,0) [below]{${\scriptstyle n}$};
\node at (3.5,-.3) [above]{${\scriptstyle n}$};
\node at (3.5,-.9) [below]{${\scriptstyle n}$};
}\,\bigg\rangle_{\!3}
&=
\sum_{s=0}^{n}
\begin{Bmatrix}
n&n&(s,s)\\
n&n&(i_k,i_k)
\end{Bmatrix}
\bigg\langle\tikz[baseline=-.6ex, scale=.5]{
\draw[triple={[line width=1.4pt, white] in [line width=2.2pt, black] in [line width=5.4pt, white]}]
(.5,.3) -- (.5,-.3);
\draw[triple={[line width=1.4pt, white] in [line width=2.2pt, black] in [line width=5.4pt, white]}]
(2,-.6) -- (3,-.6);
\draw[->-=.5] (.5,.3) to[out=west, in=west] (.5,.9) -- (1.5,.9);
\draw[-<-=.5] (.5,-.3) to[out=west, in=west] (.5,-.9) -- (1.5,-.9);
\draw[-<-=.8] (.5,.3) -- (1.5,.3);
\draw[->-=.8] (.5,-.3) -- (1.5,-.3);
\draw (1.5,.9) -- (3.5,.9);
\draw (1.5,.3) -- (3.5,.3);
\draw (1.5,-.3) to[out=east, in=north west] (2,-.6);
\draw (1.5,-.9) to[out=east, in=south west] (2,-.6);
\draw[->-=.8] (3,-.6) to[out=north east, in=west] (3.5,-.3);
\draw[-<-=.8] (3,-.6) to[out=south east, in=west] (3.5,-.9);
\node at (.5,0) [left]{${\scriptstyle s}$};
\node at (2.5,-.6) [above]{${\scriptstyle i_{k+1}}$};
\node at (0,.6) {${\scriptstyle n}$};
\node at (0,-.6) {${\scriptstyle n}$};
\node at (1,.6) {${\scriptstyle n}$};
\node at (1,-.6) {${\scriptstyle n}$};
\node at (3.5,0) {${\scriptstyle n}$};
\node at (3.5,-.6) {${\scriptstyle n}$};
}\,\bigg\rangle_{\!3}\label{eq2}\\
&=
\begin{Bmatrix}
n&n&(i_{k+1},i_k+1)\\
n&n&(i_k,i_k)
\end{Bmatrix}
\bigg\langle\tikz[baseline=-.6ex, scale=.5]{
\draw[triple={[line width=1.4pt, white] in [line width=2.2pt, black] in [line width=5.4pt, white]}]
(2,-.6) -- (3,-.6);
\draw[-<-=.5] (1.5,-.3) to[out=west, in=west] (1.5,.3);
\draw[->-=.5] (1.5,-.9) to[out=west, in=west] (1.5,.9);
\draw (1.5,.9) -- (3.5,.9);
\draw (1.5,.3) -- (3.5,.3);
\draw (1.5,-.3) to[out=east, in=north west] (2,-.6);
\draw (1.5,-.9) to[out=east, in=south west] (2,-.6);
\draw[->-=.8] (3,-.6) to[out=north east, in=west] (3.5,-.3);
\draw[-<-=.8] (3,-.6) to[out=south east, in=west] (3.5,-.9);
\node at (2.5,-.6) [above]{${\scriptstyle i_{k+1}}$};
\node at (1,0) [left]{${\scriptstyle n}$};
\node at (2.5,.6) {${\scriptstyle n}$};
\node at (3.5,0) {${\scriptstyle n}$};
\node at (3.5,-.6) {${\scriptstyle n}$};
}\,\bigg\rangle_{\!3}.\notag
\end{align}
First, 
we apply (\ref{eq1}) to boxed $2a_k$ for all $0\leq k\leq l$ of the two bridge link diagram. 
Next, we repeatedly apply (\ref{eq2}) to the leftmost doubled edge. 
Finally, 
we obtain $\theta(n,n,(i_l,i_l))$.
\end{proof}

\section{The tail of the $\mathfrak{sl}_3$ colored Jones polynomial for $T(2,2m)$}
We discuss a stability of coefficients of the $\mathfrak{sl}_3$ colored Jones polynomial for the torus link $T(2,2m)$ in this section for positive integer $m$. 
Two explicit formulas of $J_{(n,0)}^{\mathfrak{sl}_3}(T(2,2m))$ are obtained from Theorem~\ref{coloredjones} and \cite[Theorem~5.7]{Yuasa16} because $T(2,2m)$ has a presentation $\left[2m\right]$ as a $2$-bridge link.
We write these two formulas below. Let $m$ be a positive integer.

From Theorem~\ref{coloredjones},
\begin{align*}
\psi_n^{(m)}(q)
&=J_{(n,0)}^{\mathfrak{sl}_3}(T(2,2m))=q^{-\frac{2m}{3}(n^2+3n)}\sum_{i=0}^n\frac{\Delta(i,i)}{\Delta(n,0)}q^{m(i^2+2i)}\\
&=q^{-\frac{2m}{3}(n^2+3n)+n}\sum_{i=0}^nq^{-2i}q^{m(i^2+2i)}\frac{(1-q^{i+1})^3(1+q^{i+1})}{(1-q)(1-q^{n+1})(1-q^{n+2})}.
\end{align*}

From \cite[Theorem~5.7]{Yuasa16},
\begin{align*}
g_n^{(m)}(q)
&=J_{(n,0)}^{\mathfrak{sl}_3}(T(2,2m))\\
&=q^{-\frac{2m}{3}(n^2+3n)+n}\sum_{0\leq k_m\leq\dots\leq k_2\leq k_1\leq n}q^{-2k_m}q^{\sum_{j=1}^m(k_i^2+2k_i)}\\
&\quad\times \frac{(q)_n^2}{(q)_{k_m}^2(q)_{n-k_1}(q)_{k_1-k_2}\dots(q)_{k_{m-1}-k_m}}
\frac{(1-q^{n+1})(1-q^{n+2})}{(1-q^{n-k_m+1})(1-q^{n-k_m+2})}.
\end{align*}

\begin{DEF}
Suppose $f(q),f_n(q)\in\mathbb{Z}[[q]]$ for $n\geq 1$. 
{\em The limit of $\{f_n(q)\}_n$ is $f(q)$}, 
denoted $\displaystyle \lim_{n\to\infty}f_n(q)=f(q) $, 
means that $f_n(q)=f(q)$ in $\mathbb{Z}[[q]]/q^{n+1}\mathbb{Z}[[q]]$ for all $n$.
\end{DEF}

We can see that the minimum degree of $J_{(n,0)}^{\mathfrak{sl}_3}(T(2,2m))$ is $q^{-\frac{2m}{3}(n^2+3n)+n}$ and consider the limit of $\Psi_n^{(m)}(q)=q^{\frac{2m}{3}(n^2+3n)-n}\psi_n(q)$ and $G_n^{(m)}(q)=q^{\frac{2m}{3}(n^2+3n)-n}g_n(q)$.
The following lemma ensures the existence of the limit.
\begin{LEM}
$\Psi_n^{(m)}(q)=\Psi_{n+1}^{(m)}(q)$ in $\mathbb{Z}[[q]]/q^{n+1}\mathbb{Z}[[q]]$.
\end{LEM}
\begin{proof}
It is clear from the expression of $\Psi_{n+1}^{(m)}(q)$.
\end{proof}

We explicitly give the limit of $\{\Psi_n^{(m)}(q)\}_n$ and $\{G_n^{(m)}\}_n$.
First, we can easily obtain the following:
\begin{equation}\label{Psi}
\lim_{n\to\infty}\Psi_n^{(m)}(q)=\sum_{i=0}^\infty q^{-2i}q^{m(i^2+2i)}\frac{(1-q^{i+1})^3(1+q^{i+1})}{1-q}.
\end{equation}
Next, 
we consider the limit of $\{G_n^{(m)}\}_n$.
In $\mathbb{Z}[[q]]/q^{n+1}\mathbb{Z}[[q]]$, 
\begin{align*}
\frac{q^{k_m^2}}{1-q^{n-k_m+1}}&=q^{k_m}(1+q^{n-k_m+1}+q^{2(n-k_m+1)}+\dots)\\
&=q^{k_m^2}+q^{n+1+k_m^2-k_m}+(\,\text{higher-order tems}\,)\\
&=q^{k_m^2}
\end{align*}
because $k_m^2-k_m\geq 0$.
In a similar way, 
we obtain 
$\frac{q^{k_m^2}}{1-q^{n-k_m+2}}=q^{k_m^2}$, $q^{k_1^2+2k_1}\frac{(q)_n}{(q)_{n-k_1}}=q^{k_1^2+2k_1}$. 
Form the above, 
we conclude that the limit of $\{G_n^{(m)}\}_n$ is 
\begin{equation}\label{G}
\lim_{n\to\infty}G_n^{(m)}(q)
=(q)_{\infty}\sum_{0\leq k_m\leq\dots\leq k_2\leq k_1}\frac{q^{-2k_m}q^{\sum_{j=1}^m(k_i^2+2k_i)}}{(q)_{k_m}^2(q)_{k_1-k_2}\dots(q)_{k_{m-1}-k_m}}
\end{equation}

Consequently, 
we have the following identity of $q$-series.

\begin{THM}\label{qseries}
\footnotesize
\[
\sum_{i=0}^\infty q^{-2i}q^{m(i^2+2i)}\frac{(1-q^{i+1})^3(1+q^{i+1})}{1-q}=(q)_{\infty}\sum_{0\leq k_m\leq\dots\leq k_2\leq k_1}\frac{q^{-2k_m}q^{\sum_{j=1}^m(k_i^2+2k_i)}}{(q)_{k_m}^2(q)_{k_1-k_2}\dots(q)_{k_{m-1}-k_m}}.\]
\normalsize
\end{THM}
The above identity is a knot-theoretical generalization of the Andrews-Gordon identities for the Ramanujan false theta function.

\begin{QUE}
Does the left-hand side of the identity of Theorem~\ref{qseries} have a representation as a well-known function like the Ramanujan false theta function?
\end{QUE}

We compute the expansion of $q$-series $\Psi^{(m)}(q)=\lim_{n\to\infty}\Psi_n^{(m)}(q)$ up to order $150$ by Mathematica~{11.0}~\cite{Mathematica} in Table~\ref{expansion}. 
It is conjectured that the collection of coefficients of the expansion of $\Psi^{(m)}(q)$ except for zeros is $\{1,{-1},{-1},1,1,1,{-1},{-1},{-1},{-1},1,1,1,1,1,\dots\}$ for $m>0$. 
Sequences of non-zero coefficients and zero coefficients appear in the expansion of $\Psi^{(m)}(q)$ alternatively. 
The $n$-th sequence of non-zero coefficients has length $4n$ and the $n$-th sequence of zero coefficients length $(2n+1)(m-2)$ for $m>0$.

\subsection*{Acknowledgment}
The author would like to express his gratitude to his adviser, Professor Hisaaki Endo, for his encouragement.

\begin{table}[p]
\begin{center}
\begin{tabular}[c]{|c|l|} \hline
$m$ & \multicolumn{1}{c|}{$\Psi^{(m)}(q)$} \\
\hline\hline
$1$ & \footnotesize{$1+O\left(q^{151}\right)$}\\ 
\hline
$2$ &  \footnotesize{$\begin{aligned}
&1-q-q^2+q^3+q^4+q^5-q^6-q^7-q^8-q^9+q^{10}+q^{11}+q^{12}+q^{13}+q^{14}-q^{15}\\
&-q^{16}-q^{17}-q^{18}-q^{19}-q^{20}+q^{21}+q^{22}+q^{23}+q^{24}+q^{25}+q^{26}+q^{27}-q^{28}-q^{29}\\
&-q^{30}-q^{31}-q^{32}-q^{33}-q^{34}-q^{35}+q^{36}+q^{37}+q^{38}+q^{39}+q^{40}+q^{41}+q^{42}+q^{43}\\
&+q^{44}-q^{45}-q^{46}-q^{47}-q^{48}-q^{49}-q^{50}-q^{51}-q^{52}-q^{53}-q^{54}+q^{55}+q^{56}+q^{57}\\
&+q^{58}+q^{59}+q^{60}+q^{61}+q^{62}+q^{63}+q^{64}+q^{65}-q^{66}-q^{67}-q^{68}-q^{69}-q^{70}-q^{71}\\
&-q^{72}-q^{73}-q^{74}-q^{75}-q^{76}-q^{77}+q^{78}+q^{79}+q^{80}+q^{81}+q^{82}+q^{83}+q^{84}+q^{85}\\
&+q^{86}+q^{87}+q^{88}+q^{89}+q^{90}-q^{91}-q^{92}-q^{93}-q^{94}-q^{95}-q^{96}-q^{97}-q^{98}-q^{99}\\
&-q^{100}-q^{101}-q^{102}-q^{103}-q^{104}+q^{105}+q^{106}+q^{107}+q^{108}+q^{109}+q^{110}+q^{111}\\
&+q^{112}+q^{113}+q^{114}+q^{115}+q^{116}+q^{117}+q^{118}+q^{119}-q^{120}-q^{121}-q^{122}-q^{123}\\
&-q^{124}-q^{125}-q^{126}-q^{127}-q^{128}-q^{129}-q^{130}-q^{131}-q^{132}-q^{133}-q^{134}-q^{135}\\
&+q^{136}+q^{137}+q^{138}+q^{139}+q^{140}+q^{141}+q^{142}+q^{143}+q^{144}+q^{145}+q^{146}+q^{147}\\
&+q^{148}+q^{149}+q^{150}+O(q^{151}) \end{aligned}$}\\
\hline
$3$ & \footnotesize{$\begin{aligned}
&1-q-q^2+q^3+q^7+q^8-q^9-q^{10}-q^{11}-q^{12}+q^{13}+q^{14}+q^{20}+q^{21}+q^{22}-q^{23}\\
&-q^{24}-q^{25}-q^{26}-q^{27}-q^{28}+q^{29}+q^{30}+q^{31}+q^{39}+q^{40}+q^{41}+q^{42}-q^{43}-q^{44}\\
&-q^{45}-q^{46}-q^{47}-q^{48}-q^{49}-q^{50}+q^{51}+q^{52}+q^{53}+q^{54}+q^{64}+q^{65}+q^{66}+q^{67}\\
&+q^{68}-q^{69}-q^{70}-q^{71}-q^{72}-q^{73}-q^{74}-q^{75}-q^{76}-q^{77}-q^{78}+q^{79}+q^{80}+q^{81}\\
&+q^{82}+q^{83}+q^{95}+q^{96}+q^{97}+q^{98}+q^{99}+q^{100}-q^{101}-q^{102}-q^{103}-q^{104}-q^{105}\\
&-q^{106}-q^{107}-q^{108}-q^{109}-q^{110}-q^{111}-q^{112}+q^{113}+q^{114}+q^{115}+q^{116}+q^{117}\\
&+q^{118}+q^{132}+q^{133}+q^{134}+q^{135}+q^{136}+q^{137}+q^{138}-q^{139}-q^{140}-q^{141}-q^{142}\\
&-q^{143}-q^{144}-q^{145}-q^{146}-q^{147}-q^{148}-q^{149}-q^{150}+O\left(q^{151}\right) \end{aligned}$}\\
\hline
$4$ & \footnotesize{$\begin{aligned}
&1-q-q^2+q^3+q^{10}+q^{11}-q^{12}-q^{13}-q^{14}-q^{15}+q^{16}+q^{17}+q^{28}+q^{29}+q^{30}\\
&-q^{31}-q^{32}-q^{33}-q^{34}-q^{35}-q^{36}+q^{37}+q^{38}+q^{39}+q^{54}+q^{55}+q^{56}+q^{57}-q^{58}\\
&-q^{59}-q^{60}-q^{61}-q^{62}-q^{63}-q^{64}-q^{65}+q^{66}+q^{67}+q^{68}+q^{69}+q^{88}+q^{89}+q^{90}\\
&+q^{91}+q^{92}-q^{93}-q^{94}-q^{95}-q^{96}-q^{97}-q^{98}-q^{99}-q^{100}-q^{101}-q^{102}+q^{103}\\
&+q^{104}+q^{105}+q^{106}+q^{107}+q^{130}+q^{131}+q^{132}+q^{133}+q^{134}+q^{135}-q^{136}-q^{137}\\
&-q^{138}-q^{139}-q^{140}-q^{141}-q^{142}-q^{143}-q^{144}-q^{145}-q^{146}-q^{147}+q^{148}+q^{149}\\
&+q^{150}+O\left(q^{151}\right) \end{aligned}$}\\
\hline
$5$ & \footnotesize{$\begin{aligned}
&1-q-q^2+q^3+q^{13}+q^{14}-q^{15}-q^{16}-q^{17}-q^{18}+q^{19}+q^{20}+q^{36}+q^{37}+q^{38}\\
&-q^{39}-q^{40}-q^{41}-q^{42}-q^{43}-q^{44}+q^{45}+q^{46}+q^{47}+q^{69}+q^{70}+q^{71}+q^{72}-q^{73}\\
&-q^{74}-q^{75}-q^{76}-q^{77}-q^{78}-q^{79}-q^{80}+q^{81}+q^{82}+q^{83}+q^{84}+q^{112}+q^{113}\\
&+q^{114}+q^{115}+q^{116}-q^{117}-q^{118}-q^{119}-q^{120}-q^{121}-q^{122}-q^{123}-q^{124}-q^{125}\\
&-q^{126}+q^{127}+q^{128}+q^{129}+q^{130}+q^{131}+O\left(q^{151}\right) \end{aligned}$}\\
\hline
$6$ & \footnotesize{$\begin{aligned}
&1-q-q^2+q^3+q^{16}+q^{17}-q^{18}-q^{19}-q^{20}-q^{21}+q^{22}+q^{23}+q^{44}+q^{45}+q^{46}\\
&-q^{47}-q^{48}-q^{49}-q^{50}-q^{51}-q^{52}+q^{53}+q^{54}+q^{55}+q^{84}+q^{85}+q^{86}+q^{87}-q^{88}\\
&-q^{89}-q^{90}-q^{91}-q^{92}-q^{93}-q^{94}-q^{95}+q^{96}+q^{97}+q^{98}+q^{99}+q^{136}+q^{137}\\
&+q^{138}+q^{139}+q^{140}-q^{141}-q^{142}-q^{143}-q^{144}-q^{145}-q^{146}-q^{147}-q^{148}-q^{149}\\
&-q^{150}+O\left(q^{151}\right)\end{aligned}$}\\
\hline
\end{tabular}
\caption{Expansion of $\Psi^{(m)}(q)$ up to order $150$ for $m=1,2,\dots,6$.}
\label{expansion}
\end{center}
\end{table}

\bibliographystyle{amsplain}
\bibliography{sl3tails}

\providecommand{\bysame}{\leavevmode\hbox to3em{\hrulefill}\thinspace}
\providecommand{\MR}{\relax\ifhmode\unskip\space\fi MR }
\providecommand{\MRhref}[2]{%
  \href{http://www.ams.org/mathscinet-getitem?mr=#1}{#2}
}
\providecommand{\href}[2]{#2}
\begin{thebibliography}{10}

\bibitem{Andrews74}
George~E. Andrews, \emph{An analytic generalization of the {R}ogers-{R}amanujan
  identities for odd moduli}, Proc. Nat. Acad. Sci. U.S.A. \textbf{71} (1974),
  4082--4085. \MR{0351985}

\bibitem{AndrewsBerndt05}
George~E. Andrews and Bruce~C. Berndt, \emph{Ramanujan's lost notebook. {P}art
  {I}}, Springer, New York, 2005. \MR{2135178}

\bibitem{Armond13}
Cody Armond, \emph{The head and tail conjecture for alternating knots}, Algebr.
  Geom. Topol. \textbf{13} (2013), no.~5, 2809--2826. \MR{3116304}

\bibitem{Armond14}
\bysame, \emph{Walks along braids and the colored {J}ones polynomial}, J. Knot
  Theory Ramifications \textbf{23} (2014), no.~2, 1450007, 15. \MR{3197051}

\bibitem{ArmondDasbach11}
Cody Armond and Oliver~T. Dasbach, \emph{Rogers-ramanujan type identities and
  the head and tail of the colored jones polynomial}, arXiv:1106.3948 (2011).

\bibitem{DasbachLin06}
Oliver~T. Dasbach and Xiao-Song Lin, \emph{On the head and the tail of the
  colored {J}ones polynomial}, Compos. Math. \textbf{142} (2006), no.~5,
  1332--1342. \MR{2264669}

\bibitem{DasbachLin07}
\bysame, \emph{A volumish theorem for the {J}ones polynomial of alternating
  knots}, Pacific J. Math. \textbf{231} (2007), no.~2, 279--291. \MR{2346497}

\bibitem{GaroufalidisLe15}
Stavros Garoufalidis and Thang T.~Q. L{\^e}, \emph{Nahm sums, stability and the
  colored {J}ones polynomial}, Res. Math. Sci. \textbf{2} (2015), Art. 1, 55.
  \MR{3375651}

\bibitem{Hajij14A}
Mustafa Hajij, \emph{The {B}ubble skein element and applications}, J. Knot
  Theory Ramifications \textbf{23} (2014), no.~14, 1450076, 30. \MR{3312619}

\bibitem{Hajij16}
\bysame, \emph{The tail of a quantum spin network}, Ramanujan J. \textbf{40}
  (2016), no.~1, 135--176. \MR{3485997}

\bibitem{Mathematica}
Wolfram~{R}esearch Inc., \emph{Mathematica, {\rm {v}ersion 11.0}}, Champaign,
  Illinois (2016).

\bibitem{KauffmanLins94}
Louis~H. Kauffman and S{\'o}stenes~L. Lins, \emph{Temperley-{L}ieb recoupling
  theory and invariants of {$3$}-manifolds}, Annals of Mathematics Studies,
  vol. 134, Princeton University Press, Princeton, NJ, 1994. \MR{1280463
  (95c:57027)}

\bibitem{Kuperberg94}
Greg Kuperberg, \emph{The quantum {$G_2$} link invariant}, Internat. J. Math.
  \textbf{5} (1994), no.~1, 61--85. \MR{1265145}

\bibitem{Kuperberg96}
\bysame, \emph{Spiders for rank {$2$} {L}ie algebras}, Comm. Math. Phys.
  \textbf{180} (1996), no.~1, 109--151. \MR{1403861}

\bibitem{McLaughlinSillsZimmer09}
James McLaughlin, Andrew~V. Sills, and Peter Zimmer, \emph{Rogers-{R}amanujan
  computer searches}, J. Symbolic Comput. \textbf{44} (2009), no.~8,
  1068--1078. \MR{2523768}

\bibitem{Morton95}
H.~R. Morton, \emph{The coloured {J}ones function and {A}lexander polynomial
  for torus knots}, Math. Proc. Cambridge Philos. Soc. \textbf{117} (1995),
  no.~1, 129--135. \MR{1297899}

\bibitem{OhtsukiYamada97}
Tomotada Ohtsuki and Shuji Yamada, \emph{Quantum {${\rm SU}(3)$} invariant of
  {$3$}-manifolds via linear skein theory}, J. Knot Theory Ramifications
  \textbf{6} (1997), no.~3, 373--404. \MR{1457194}

\bibitem{Roberts94}
Justin Roberts, \emph{Skeins and mapping class groups}, Math. Proc. Cambridge
  Philos. Soc. \textbf{115} (1994), no.~1, 53--77. \MR{1253282}

\bibitem{Rogers1894}
L.~J. Rogers, \emph{Second {M}emoir on the {E}xpansion of certain {I}nfinite
  {P}roducts}, Proc. London Math. Soc. \textbf{S1-25} (1894), no.~1, 318.
  \MR{1576348}

\bibitem{TuraevViro92}
V.~G. Turaev and O.~Ya. Viro, \emph{State sum invariants of {$3$}-manifolds and
  quantum {$6j$}-symbols}, Topology \textbf{31} (1992), no.~4, 865--902.
  \MR{1191386}

\bibitem{Yuasa16}
Wataru Yuasa, \emph{The {$\mathfrak{sl}_3$} colored {J}ones polynomials for
  {$2$}-bridge links}, arXiv:1609.07289 (2016).

\end{thebibliography}
\end{document}